\documentclass[10pt, reqno]{amsart}  
\usepackage{amsmath, amsthm, amssymb, eucal}
\usepackage{graphicx, xcolor}
\usepackage{multicol}
\usepackage[font=footnotesize,labelfont=bf]{caption}
\usepackage{subcaption}

\usepackage{hyperref}
\usepackage{cite}
\usepackage{enumitem}
\setlist[enumerate]{itemsep=0.5\baselineskip}
\setlist[itemize]{itemsep=0.5\baselineskip}

\begin{document}
\allowdisplaybreaks

\numberwithin{equation}{section}
\theoremstyle{plain}
\newtheorem{Proposition}[equation]{Proposition}
\newtheorem{Corollary}[equation]{Corollary}
\newtheorem*{Corollary*}{Corollary}
\newtheorem{Theorem}[equation]{Theorem}
\newtheorem*{Theorem*}{Theorem}
\newtheorem{Lemma}[equation]{Lemma}
\theoremstyle{definition}
\newtheorem{Definition}[equation]{Definition}
\newtheorem{Conjecture}[equation]{Conjecture}
\newtheorem{Example}[equation]{Example}
\newtheorem{Remark}[equation]{Remark}

\def\C{\mathbb{C}}
\def\R{\mathbb{R}}
\def\D{\mathbb{D}}
\def\T{\mathbb{T}}
\def\N{\mathbb{N}}
\def\Z{\mathbb{Z}}

\newcommand{\0}{{\textcolor{lightgray}0}}
\newcommand{\M}{\mathsf{M}}
\renewcommand{\P}{\mathsf{PI}}
\newcommand{\U}{\mathsf{U}}
\newcommand{\K}{\mathcal{K}}
\newcommand{\minimatrix}[4]{\begin{bmatrix} #1 & #2 \\ #3 & #4 \end{bmatrix}  }
\newcommand{\twovector}[2]{\begin{bmatrix} #1 \\ #2 \end{bmatrix} }
\newcommand{\threevector}[3]{\begin{bmatrix} #1 \\ #2 \\ #3 \end{bmatrix} }
\newcommand{\fourrowvector}[4]{[#1\,\,#2\,\,#3\,\,#4]}
\newcommand{\vecspan}{\operatorname{span}}
\newcommand{\colspace}{\operatorname{colspace}}
\newcommand{\nullspace}{\operatorname{nullspace}}
\newcommand{\nullity}{\operatorname{nullity}}
\newcommand{\norm}[1]{\| #1 \|}
\newcommand{\inner}[1]{\langle #1 \rangle}
\newcommand{\tr}{\operatorname{tr}}
\newcommand{\diag}{\operatorname{diag}}
\newcommand{\rank}{\operatorname{rank}}
\renewcommand{\labelenumi}{(\alph{enumi})}
\newcommand{\comment}[1]{\marginpar{\tiny\color{blue} $\bullet$ #1}}
\renewcommand{\Re}{\operatorname{Re}}
\renewcommand{\Im}{\operatorname{Im}}
\newcommand{\co}{\operatorname{conv}}

\renewcommand{\le}{\leqslant}
\renewcommand{\leq}{\leqslant}
\renewcommand{\ge}{\geqslant}
\renewcommand{\geq}{\geqslant}
\renewcommand{\setminus}{\backslash}
\newcommand{\ran}{\operatorname{ran}}

\renewcommand{\subset}{\subseteq}
\renewcommand{\supset}{\supseteq}

\renewcommand{\vec}[1]{{\bf #1}}

\title[Partially isometric matrices]{Partially isometric matrices: a brief and selective survey}

 \author[S.R.~Garcia]{Stephan Ramon Garcia}
   \address{Department of Mathematics,
            Pomona College,
            Claremont, California
            91711, USA}
    \email{Stephan.Garcia@pomona.edu}
    \thanks{The first author was partially supported by NSF grant DMS-1800123.}

 \author[M.O.~Patterson]{Matthew Okubo Patterson}
   \address{Department of Mathematics,
            Pomona College,
            Claremont, California
            91711, USA}
 \email{matthewpatterson477@gmail.com}

\author[W.T.~Ross]{William T. Ross}
	\address{Department of Mathematics and Computer Science, University of Richmond, Richmond, VA 23173, USA}
	\email{wross@richmond.edu}
	


\begin{abstract}
We survey a variety of results about partially isometric matrices.  We focus primarily on results that are distinctly finite-dimensional.  For example, we cover a recent solution to the similarity problem for partial isometries.  We also discuss the unitary similarity problem and several other results.
\end{abstract}
	
\maketitle

\section{Introduction}

This paper is a selective survey about partially isometric matrices.
These matrices are characterized by the equation $AA^*A = A$, in which $A^*$ denotes the conjugate transpose of $A$.  We refer to such a matrix as a \emph{partial isometry}.
Much of this material dates back to early work of 
Erd\'elyi \cite{Erd,MR0255557,ErdelyiProduct}, Halmos \& McLaughlin \cite{Halmos}, and Hearon \cite{MR0225790}, among others.
Some of our results will be familiar to many readers.  Others are more recent or perhaps not so well known.

The study of partial isometries on infinite-dimensional spaces is much richer and more difficult.
A proper account of the infinite-dimensional setting would occupy a large volume and we  
therefore restrict ourselves here to the finite-dimensional case.
For the sake of simplicity, and because we are interested in topics such as similarity and unitary similarity, 
we further narrow our attention to square matrices.  However, many of the following results hold for nonsquare matrices if the indices and subscripts
are adjusted appropriately.

This survey is organized as follows.
Section \ref{Section:Preliminaries} introduces the basic properties of partial isometries.
In particular, the connection between partial isometries, orthogonal projections, and subspaces
is considered.
Section \ref{Section:Algebraic} covers the algebraic structure of partial isometries.  For example, we consider the singular value and polar decompositions, the Moore--Penrose pseudoinverse, and products of partial isometries.
In Section \ref{Section:Similar} we study the similarity problem for partial isometries and characterize their spectra and Jordan canonical forms.
Section \ref{Section:UnitarySimilar} concerns various topics connected to unitary similarity.
For example, partial isometric extensions of contractions, the Liv\v{s}ic characteristic function, and the
Halmos--McLaughlin characterization of defect-one partial isometries are covered.
We conclude in Section \ref{Section:Compressed} with a brief treatment of the compressed shift operator,
a concrete realization of certain partial isometries in terms of operators on spaces of rational functions.

\subsection*{Notation}
In what follows, 
$\M_{m \times n}$ denotes the set of $m \times n$ complex matrices.  We write $\M_n$ for the set of $n \times n$ complex matrices.
A convenient shorthand for the $n \times n$ diagonal matrix with diagonal entries $\lambda_1,\lambda_2,\ldots,\lambda_n$ is 
$\diag(\lambda_1,\lambda_2,\ldots,\lambda_n)$.
The spectrum of $A \in \M_n$ (the set of eigenvalues of $A$) is denoted $\sigma(A)$ and its characteristic polynomial is
$p_A(z) = \det (z I -A)$.  The open unit disk $|z| < 1$ and unit circle $|z| = 1$ are denoted $\D$ and $\T$, respectively.
We write $I_n$ and $0_n$ for the $n \times n$ identity and zero matrices, respectively.  Occasionally $0$ denotes
a zero matrix whose size is to be inferred from context.
Boldface letters, such as $\vec{x}$, denote column vectors.  Zero vectors are written as $\vec{0}$ and their lengths 
determined from context.  A row vector is the transpose $\vec{x}^{\mathsf{T}}$ of a column vector.
The range (or column space) and kernel (or nullspace) of $A \in \M_n$ are denoted $\ran A$ and $\ker A$,
respectively.  By $\|A\|$ we mean the operator norm of $A$, the maximum of $\|A \vec{x}\|$ for $\|\vec{x}\| = 1$.

\section{Preliminaries}\label{Section:Preliminaries}

Although partially isometric matrices enjoy several equivalent definitions,
we choose a distinctively algebraic approach because of its intrinsic nature.  
This suits the matrix-theoretic perspective adopted in this article and permits us to phrase things 
mostly in terms of matrices (as opposed to subspaces).

\begin{Definition}\label{Definition:PI}
$A\in \M_n$ is a \emph{partially isometric matrix} (or \emph{partial isometry}) if $AA^*A = A$.
\end{Definition}

The preceding definition is concise.  However, it does not provide much intuition about what a partial isometry is,
although it does hint at potential relationships with unitary matrices, orthogonal projections, and the Moore--Penrose pseudoinverse.  All of these suggestions
are fruitful and relevant.   

Before proceeding, we require a brief review of two important topics.
We say that $A,B \in \M_n$ are \emph{unitarily similar} (denoted $A \cong B$) 
if there is a unitary $U \in \M_n$ such that $A = UBU^*$.  
As the notation $\cong$ suggests, unitary similarity is an equivalence relation on $\M_n$.
Recall that $P \in \M_n$ is an \emph{orthogonal projection} if
$P$ is Hermitian and idempotent ($P=P^*$ and $P^2 = P$).
The spectrum of an orthogonal projection
is contained in $\{0,1\}$ and the spectral theorem ensures that
$P \cong I_r \oplus 0_{n-r}$, in which $r = \rank P$.  We permit $P = 0$ and $P= I$ in the degenerate cases
$r = 0$ and $r=n$, respectively.  In particular, 
$\C^n = \ker P \oplus \ran P$,
in which $\ker P$ and $\ran P$, the eigenspaces corresponding 
to $0$ and $1$, respectively, are orthogonal. 

We now investigate several consequences of Definition \ref{Definition:PI} and identify a few distinguished
classes of partial isometries. 

\begin{Proposition}\label{Proposition:Basic}
\hfill
\begin{enumerate}[leftmargin=*]
\item $A$ is a partial isometry if and only if $A^*$ is a partial isometry.

\item If $P \in \M_n$ is an orthogonal projection, then $P$ is a partial isometry.

\item If $A \in \M_n$ is a partial isometry and $U,V \in \M_n$ are unitary, then $UAV$ is a partial isometry.

\item A matrix that is unitarily similar to a partial isometry is a partial isometry.

\item If $U \in \M_n$ is unitary, then $U$ is a partial isometry.

\item If $A$ is a normal partial isometry, then $A$ is unitarily similar to the
direct sum of a zero matrix and a unitary matrix (either factor may be omitted).

\item An invertible partial isometry is unitary.
\end{enumerate}
\end{Proposition}

\begin{proof}
\smallskip\noindent(a) $AA^*A = A$ and $A^*(A^*)^*A^* = A^*$ are adjoints of each other.

\smallskip\noindent(b) If $P \in \M_n$ is an orthogonal projection, then
$PP^*P = P^3 = P$.

\smallskip\noindent(c) If $A \in \M_n$ is a partial isometry and $B= UAV$, in which $U,V \in \M_n$ are unitary,
then $BB^*B = (UAV)(UAV)^*(UAV) = UAA^*AV = UAV = B$.  

\smallskip\noindent(d) Let $V = U^*$ in (c).

\smallskip\noindent(e) Let $A = V = I$ in (c).

\smallskip\noindent(f) Suppose that $A\in \M_n$ is a normal partial isometry.
In light of the spectral theorem and (d), we may assume that $A$ is diagonal.  Then
$A = AA^*A$ implies that $\lambda = \lambda|\lambda|^2$ for all $\lambda \in \sigma(A)$.
Thus, $\sigma(A) \subseteq \{0\} \cup \T$ and hence $A$ is the 
direct sum of a zero matrix and a unitary matrix (either factor may be omitted).

\smallskip\noindent(g) If $A \in \M_n$ is invertible and $AA^*A = A$, then $A^*A = I$.  Thus, $A$ is unitary.
\end{proof}

An important relationship between partial isometries and orthogonal projections
is contained in the following theorem.

\begin{Theorem}\label{Theorem:PI}
For $A \in \M_{n}$ the following conditions are equivalent.
\begin{enumerate}
\item $A$ is a partial isometry.
\item $A^*A$ is an orthogonal projection (in fact, the projection onto $(\ker A)^{\perp}$).
\item $AA^*$ is an orthogonal projection (in fact, the projection onto $\ran A$). 
\end{enumerate}
\end{Theorem}

\begin{proof}
\smallskip\noindent (a) $\Rightarrow$ (b) If $AA^*A = A$, then $(A^*A)^2 = A^*A$.  Since
$A^*A$ is selfadjoint and idempotent, it is an orthogonal projection.  Since\footnote{First observe that $\ker A \subseteq \ker A^*A$.  For the converse, note that if $\vec{x}\in \ker A^*A$, then $\norm{ A \vec{x}}^2 = \inner{ A\vec{x}, A \vec{x}} = \inner{A^*A\vec{x}, \vec{x}} = 0$ and hence $\vec{x} \in \ker A$.  Thus, $\ker A^*A \subseteq \ker A$.} 
$\ker A^*A = \ker A$, it follows that $A^*A$ is the orthogonal projection onto $(\ker A)^{\perp}$.

\smallskip\noindent (b) $\Rightarrow$ (a) If $A^*A$ is an orthogonal projection, then it is the orthogonal projection onto $(\ker A)^{\perp}$.
For $\vec{x} \in \ker A$, we have $A \vec{x} = \vec{0} = AA^*A\vec{x}$.
If $\vec{x} \in (\ker A)^{\perp}$, then $\vec{x} = A^*A\vec{x}$ and hence
$A\vec{x} = A(A^*A\vec{x}) = (AA^*A)\vec{x}$.
Thus, $A = AA^*A$.

\smallskip\noindent (b) $\Leftrightarrow$ (c) 
Proposition \ref{Proposition:Basic} and the equivalence (a) and (b) ensure that
$A^*A$ is an orthogonal projection $\Leftrightarrow$ $A$ is a partial isometry
$\Leftrightarrow$ $A^*$  is a partial isometry  $\Leftrightarrow$ $(A^*)^*(A^*) = AA^*$ is a partial isometry.
\end{proof}

\begin{Corollary}\label{Corollary:Norm}
If $A \in \M_n$ is a partial isometry and $A \neq 0$, then $\norm{A} = 1$.
\end{Corollary}

\begin{proof}
If $A \in \M_n$ is a partial isometry and $A \neq 0$, 
then $\norm{A}^2 = \norm{A^*A} = 1$ since $A^*A$ is a nonzero orthogonal projection.
\end{proof}

\begin{Example}\label{Example:ThreePIs}
The matrices
\begin{equation*}
	A=
    \begin{bmatrix}
        0 & 0\\[2pt]
        \frac{ \sqrt{3} }{2}  & \frac{1}{2}
    \end{bmatrix}
    ,\qquad
   B=  \minimatrix{0}{1}{0}{0}   ,
    \quad \text{and} \quad
    C=\frac{1}{3}
    \begin{bmatrix}
 2 & -1 & 0 \\
 2 & 2 & 0 \\
 -1 & 2 & 0 \\
    \end{bmatrix}
\end{equation*}
are partial isometries since
\begin{equation*}
A^*A= 
\begin{bmatrix}
\frac{3}{4} & \frac{\sqrt{3}}{4} \\[2pt]
\frac{\sqrt{3}}{4} & \frac{1}{4}
\end{bmatrix},\qquad
B^*B = \minimatrix{0}{\0}{\0}{1},\qquad
\quad \text{and} \quad 
C^*C = 
\begin{bmatrix}
1 & \0 & \0 \\
\0 & 1 & \0 \\
\0 & \0 & 0 \\
\end{bmatrix}
\end{equation*}
are orthogonal projections.  Also note that
\begin{equation*}
AA^* = \minimatrix{0}{\0}{\0}{1}, \qquad
BB^* = \minimatrix{1}{\0}{\0}{0},
\quad \text{and} \quad
CC^* = \frac{1}{9}
\begin{bmatrix}
 5 & 2 & -4 \\
 2 & 8 & 2 \\
 -4 & 2 & 5 \\
\end{bmatrix}
\end{equation*}
are orthogonal projections.    
\end{Example}

\begin{Example}\label{Example:N0}
If $N \in \M_{n \times r}$ has orthonormal columns, then $A = [N\,\,0] \in \M_n$
is a partial isometry since $N^*N = I \in \M_r$ and hence
\begin{equation*}
A^*A = \twovector{N^*}{0}[N\,\,0]=\minimatrix{N^*N}{0}{0}{0} = \minimatrix{I_r}{0}{0}{0_{n-r}}
\end{equation*}
is an orthogonal projection.
\end{Example}

\begin{Definition}
If $A \in \M_n$ is a partial isometry, then $(\ker A)^{\perp}$
is the \emph{initial space} of $A$ and $\ran A$ is the \emph{final space} of $A$.
\end{Definition}

If $A \in \M_n$ is a partial isometry, then $A^*A$ and $AA^*$ are orthogonal projections.
We can be more specific:  they are the orthogonal projections onto the initial and final spaces of $A$, respectively.  
The following proposition indicates the origin of the term ``partial isometry.''

\begin{Proposition}\label{Proposition:Isometric}
If $A \in \M_n$ is a partial isometry, then $A$ maps $(\ker A)^{\perp}$ isometrically onto $\ran A$.
\end{Proposition}

\begin{proof}
If $A \in \M_n$ is a partial isometry, then $A^*A$ is the orthogonal projection
onto $(\ker A)^{\perp}$ (Theorem \ref{Theorem:PI}).  For $\vec{x} \in (\ker A)^{\perp}$, 
we have $\norm{A \vec{x}}^2 = \inner{A \vec{x}, A \vec{x}} = \inner{A^*A\vec{x}, \vec{x}} = 
\inner{ \vec{x}, \vec{x}} = \norm{ \vec{x}}^2$.
Thus, $A$ maps $(\ker A)^{\perp}$ isometrically into $\ran A$.  Since
\begin{equation*}
\dim \ker A + \dim (\ker A)^{\perp} = n = \dim \ker A + \dim \ran A,
\end{equation*}
we see that $\dim \ran A = \dim (\ker A)^{\perp}$, so
the image of $(\ker A)^{\perp}$ under $A$ is $\ran A$.
\end{proof}

\begin{Example}
For the partial isometries in Example \ref{Example:ThreePIs},
\begin{align*}
(\ker A)^{\perp} &= \vecspan\left\{ \twovector{ 1}{-\sqrt{3}} \right\} ,
& \ran A &= \vecspan\left\{ \twovector{0}{1} \right\} , \\
(\ker B)^{\perp} &= \vecspan\left\{ \twovector{ 0}{1} \right\} ,
& \ran B &= \vecspan\left\{ \twovector{1}{0} \right\},\\
(\ker C)^{\perp} &= \vecspan\left\{ \threevector{1}{0}{0},\threevector{0}{1}{0} \right\} ,
& \ran C &= \vecspan\left\{ \threevector{2}{2}{-1}, \threevector{-1}{2}{2} \right\}.
\end{align*}
\end{Example}

\section{Algebraic properties and factorizations}\label{Section:Algebraic}
In this section we survey a few algebraic results about partial isometries.
Section \ref{Section:SVD} concerns singular value decompositions of a partial isometry.
A characterization of partial isometries in terms of the Moore--Penrose pseudoinverse is
discussed in Section \ref{Section:Pseudo}.  The role of partial isometries in the polar decomposition of a square matrix is covered in 
Section \ref{Section:Polar}.  We wrap up with a study of products of partial isometries in Section \ref{Section:Products}.

\subsection{Singular value decomposition}\label{Section:SVD}
A \emph{singular value decomposition} (SVD) of $A \in \M_n$ is a factorization of the form
$A = U \Sigma V^*$, in which $U,V \in \M_n$ are unitary and 
\begin{equation*}
\Sigma = \diag(\sigma_1,\sigma_2,\ldots,\sigma_n)
\end{equation*}
with
\begin{equation*}
\sigma_1 \geq \sigma_2 \geq \cdots \sigma_n \geq 0;
\end{equation*}
see \cite[Thm.~14.1.4]{GarciaHorn}.  
Singular value decompositions always exist, but they are never unique (for example, replace $U,V$ with $-U, -V$, respectively).
For a general $A \in \M_{m \times n}$, a similar decomposition holds with $U \in \M_m$, $V \in \M_n$, and $\Sigma \in \M_{m \times n}$.

The nonnegative numbers $\sigma_i$ above are the \emph{singular values} of $A$; they are the square roots of the eigenvalues of
$A^*A$ and $AA^*$ since
\begin{equation*}
V^*(A^*A)V = \Sigma^2 = U^*(AA^*)U.
\end{equation*}
In particular, $\sigma_1 = \norm{A}$ and $\sigma_{r+1} = \cdots = \sigma_n = 0$, in which $\rank A = r$.
For $r=1,2,\ldots,n-1$, define
\begin{equation}\label{eq:Xr}
X_r = I_r \oplus 0_{n-r}.
\end{equation}
By convention, we let $\Sigma_0 = 0$ and $\Sigma_n = I$.  The following theorem characterizes
singular value decompositions of partial isometries.

\begin{Theorem}\label{Theorem:SVD}
For $A \in \M_{n}$ with $\rank A = r$, the following are equivalent.
\begin{enumerate}
\item $A \in \M_n$ is a partial isometry.
\item $A = U X_r V^*$ for some unitary $U,V \in \M_n$.
\item $U^*AV$ is a partial isometry for some unitary $U,V \in \M_n$.
\end{enumerate}
\end{Theorem}

\begin{proof}
(a) $\Rightarrow$ (b) 
If $A \in \M_n$ is a partial isometry with singular value decomposition
$A = U \Sigma V^*$, then $A = AA^*A$ implies 
$\Sigma^3 = \Sigma$.  Thus, the diagonal entries of $\Sigma$ belong to $\{0,1\}$.
Since $\rank A = \rank \Sigma$, we have $\Sigma = X_r$, in which $r = \rank A$.

\smallskip\noindent (b) $\Rightarrow$ (c)
Since $X_r$ is an orthogonal projection, this follows from Proposition \ref{Proposition:Basic}.

\smallskip\noindent (c) $\Rightarrow$ (a)
If $B = U^*AV$ is a partial isometry for some unitary $U,V \in \M_n$,
then $A = UBV^*$ is a partial isometry by Proposition \ref{Proposition:Basic}.
\end{proof}

\begin{Example}\label{Example:R2PI}
The rank-$2$ partial isometry
\begin{equation*}
\begin{bmatrix}
 1 & \0 & \0 \\[2pt]
 \0 & \frac{\sqrt{3}}{2} & \0 \\[2pt]
 \0 & \frac{1}{2} & \0 \\
\end{bmatrix}
\end{equation*}
has singular value decomposition
\begin{equation*}
\begin{bmatrix}
 \0 & 1 & \0 \\[2pt]
 \frac{\sqrt{3}}{2} & \0 & -\frac{1}{2} \\[2pt]
 \frac{1}{2} & \0 & \frac{\sqrt{3}}{2} \\
\end{bmatrix} 
\begin{bmatrix}
 1 & \0 & \0 \\
 \0 & 1 & \0 \\
 \0 & \0 & 0 \\
\end{bmatrix} 
\begin{bmatrix}
 \0 & 1 & \0 \\
 1 & \0 & \0 \\
 \0 & \0 & 1 \\
\end{bmatrix}.
\end{equation*}
\end{Example}

The characterization of partial isometries in terms of the singular value decomposition
leads to a standard presentation of a partial isometry, up to unitary similarity.

\begin{Theorem}\label{Theorem:N}
For $A \in \M_{n}$ with $\rank A = r$ the following are equivalent.
\begin{enumerate}
\item $A$ is a partial isometry. 
\item $A \cong [N\,\,0]$, in which $N \in \M_{n\times r}$ has orthonormal columns.
\item $A \cong \minimatrix{B}{0}{C}{0}$, in which $B \in \M_r$, $C\in \M_{(n-r)\times r}$,
and $B^*B + C^*C = I_r$.
\end{enumerate}
\end{Theorem}

\begin{proof}
(a) $\Rightarrow$ (b) Let $A \in \M_n$ be a partial isometry.
Theorem \ref{Theorem:SVD} ensures that 
$A= UX_rV^*$ for some unitary $U,V \in \M_n$.
Thus, $V^*AV = (V^*U)X_r = [N \,\,0]$, in which $N \in \M_{n \times r}$ is comprised
of the first $r$ columns (necessarily orthonormal) of the unitary matrix $V^*U$.  

\smallskip\noindent(b) $\Rightarrow$ (c)
Suppose that $A \cong [N\,\,0]$, in which $N \in \M_{n\times r}$ has orthonormal columns.
Then Proposition \ref{Proposition:Basic} ensures that 
\begin{equation*}\qquad\qquad\qquad
[N\,\,0] =  \minimatrix{B}{0}{C}{0},\qquad B \in \M_r, \,C\in \M_{(n-r)\times r},
\end{equation*}
is a partial isometry.  Since $N^*N = I_r$,
\begin{equation}\label{eq:CompAbovePI}
\minimatrix{I_r}{0}{0}{0_{n-r}} = \twovector{N^*}{0}[N\,\,0]
= \minimatrix{B^*}{C^*}{0}{0} \minimatrix{B}{0}{C}{0}
= \minimatrix{B^*B+C^*C}{0}{0}{0}
\end{equation}
and hence $B^*B+C^*C = I_r$.

\smallskip\noindent(c) $\Rightarrow$ (a)
The computation \eqref{eq:CompAbovePI} and Theorem \ref{Theorem:PI}
ensure that $A$ is unitarily similar to a partial isometry.
\end{proof}

\begin{Example}
The matrix
\begin{equation*}
A = \frac{1}{9}
\begin{bmatrix}
 8 & 2 & 2 \\
 2 & 5 & -4 \\
 -2 & 4 & -5 \\
\end{bmatrix}
\end{equation*}
is a partial isometry.  Indeed, $A = U X_2 V^*$, in which
\begin{equation*}
U = 
\begin{bmatrix}
 \frac{2}{\sqrt{5}} & \frac{2}{3 \sqrt{5}} & \frac{1}{3} \\[4pt]
 0 & \frac{\sqrt{5}}{3} & -\frac{2}{3} \\[3pt]
 -\frac{1}{\sqrt{5}} & \frac{4}{3 \sqrt{5}} & \frac{2}{3} \\
\end{bmatrix}
\qquad \text{and} \qquad
V = 
\begin{bmatrix}
 \frac{2}{\sqrt{5}} & \frac{2}{3 \sqrt{5}} & -\frac{1}{3} \\[4pt]
 0 & \frac{\sqrt{5}}{3} & \frac{2}{3} \\[3pt]
 \frac{1}{\sqrt{5}} & -\frac{4}{3 \sqrt{5}} & \frac{2}{3} \\
\end{bmatrix}
\end{equation*}
are unitary.  Following the proof of (a) $\Rightarrow$ (b) in Theorem \ref{Theorem:N} we find
\begin{equation*}
V^*AV = 
\begin{bmatrix}
 \frac{3}{5} & \frac{8}{15} & \0 \\[4pt]
 \frac{8}{15} & \frac{13}{45} & \0 \\[4pt]
 -\frac{4}{3 \sqrt{5}} & \frac{16}{9 \sqrt{5}} & \0 \\
\end{bmatrix} = [N\,\,0],
\end{equation*}
in which $N$ has orthonormal columns.  Moreover, $B^*B + C^*C = I_2$, in which
\begin{equation*}
B = 
\begin{bmatrix}
 \frac{3}{5} & \frac{8}{15} \\[4pt]
 \frac{8}{15} & \frac{13}{45} \\
\end{bmatrix}
\qquad \text{and} \qquad
C =  \left[-\tfrac{4}{3 \sqrt{5}} \,\,\,\, \tfrac{16}{9 \sqrt{5}}  \right],
\end{equation*}
as suggested by Theorem \ref{Theorem:N}c.
\end{Example}

\subsection{Pseudoinverses}\label{Section:Pseudo}
Let $A\in \M_n$ with $\rank A = r$ and let $\sigma_1 \geq \sigma_2 \geq \cdots \geq \sigma_r$
be the nonzero singular values of $A$.  Let $A = U \Sigma V^*$ be a singular
value decomposition of $A$, in which
\begin{equation*}
\Sigma = \diag(\sigma_1,\sigma_2,\ldots,\sigma_r,0,0,\ldots,0).
\end{equation*}
Then
\begin{equation*}
\Sigma^+ = \diag(\sigma_1^{-1},\sigma_2^{-1},\ldots,\sigma_r^{-1},0,0,\ldots,0)
\end{equation*}
satisfies 
\begin{equation*}
\Sigma \Sigma^+ = \Sigma^+ \Sigma = X_r,
\end{equation*}
in which $X_r = I_r \oplus 0_{n-r}$, as defined in \eqref{eq:Xr}.
The \emph{pseudoinverse} of $A$ is 
\begin{equation*}
A^+ = V\Sigma^+ U^*,
\end{equation*}
which satisfies
\begin{enumerate}
\item $A A^{+} A = A$,
\item $A^{+} A A^{+} = A^{+}$,
\item $(A A^{+})^* = A A^{+}$, and
\item $(A^{+} A)^* = A^{+} A$.
\end{enumerate}
In particular, $A^{+} = A^{-1}$ if $A$ is invertible.  
The matrix $A^+$ is uniquely determined by the conditions (a)-(d) above and is often alternately
referred to as the \emph{Moore--Penrose generalized inverse} of $A$.
The pseudoinverse satisfies $(AB)^+ = B^+ A^+$ for $A,B \in \M_n$.

\begin{Theorem}\label{Theorem:MP}
$A \in \M_n$ is a partial isometry if and only if $A^+ = A^*$.
\end{Theorem}

\begin{proof}
Let $A \in \M_n$ have singular value decomposition $A = U \Sigma V^*$.  
If $A$ is a partial isometry, then $\Sigma = X_r$, in which $\rank A=r$ (Theorem \ref{Theorem:SVD}).
Thus, $A^+ = V \Sigma^+ U^* = V X_r U^* = A^*$.
Conversely, suppose that $A^+ = A^*$.  Then
$V\Sigma^+ U^* = V\Sigma U^*$ and hence $\Sigma^+ = \Sigma$.
The definition of $\Sigma^+$ ensures that 
each nonzero singular value of $A$ is $1$ and hence $A = UX_rV^*$
is a partial isometry (Theorem \ref{Theorem:SVD}).
\end{proof}

\subsection{Polar decomposition}\label{Section:Polar}
The singular value decomposition leads to a matrix analogue of the polar
form of a complex number $z$, in which partial isometries play a critical role.
We first consider a closely-related factorization of partial isometries.

\begin{Theorem}\label{Theorem:Polar}
For $A \in \M_{n}$ the following are equivalent.
\begin{enumerate}
\item $A$ is a partial isometry.
\item $A = W P$, in which $P$ is an orthogonal projection and $W$ is unitary. 
\item $A = Q W$, in which $Q$ is an orthogonal projection and $W$ is unitary.
\end{enumerate}
\end{Theorem}

\begin{proof}
(a) $\Rightarrow$ (b)
Let $A \in \M_n$ be a partial isometry with singular value decomposition
$A = U X_r V^*$ (Theorem \ref{Theorem:SVD}).  Then 
$A = WP$, in which $W = UV^*$ is unitary and
$P = VX_r V^*$ is an orthogonal projection.

\smallskip\noindent(b) $\Rightarrow$ (c)
Let $A = WP$, in which $W$ is unitary $P$ is an orthogonal projection.  Then
$A = QW$, in which $Q= WPW^*$ is an orthogonal projection.

\smallskip\noindent(c) $\Rightarrow$ (a)
If $A = QW$, in which $Q$ is an orthogonal projection and $W$ is unitary, then
$AA^*A = (QW)(QW)^*(QW) = QWW^*Q^*QW = Q^3W = QW = A$
since $Q$ is Hermitian and idempotent.
\end{proof}

Theorem \ref{Theorem:Polar} permits one to extend a non-unitary
partial isometry to a unitary matrix.  If $A$ is a partial isometry and $A = WP$, in which $W$ is 
unitary and $P = A^*A$ is an orthogonal projection, then $W$ agrees with
$A$ on the initial space $(\ker A)^{\perp}$ and acts on $\ker A$ such that 
$\| W \vec{x} \| = \| \vec{x}\|$ for all $\vec{x} \in \C^n$.  We regard $W$ as a unitary extension of $A$.

\begin{Example}
The rank-$2$ partial isometry 
from Example \ref{Example:R2PI} factors as
\begin{equation*}
\underbrace{
\begin{bmatrix}
 1 & \0 & \0 \\[2pt]
 \0 & \frac{\sqrt{3}}{2} & \0 \\[3pt]
 \0 & \frac{1}{2} & \0 \\
\end{bmatrix}
}_A
= 
\underbrace{
\begin{bmatrix}
 1 & \0 & \0 \\[2pt]
 \0 & \frac{\sqrt{3}}{2} & -\frac{1}{2} \\[3pt]
 \0 & \frac{1}{2} &  \frac{\sqrt{3}}{2}\\
\end{bmatrix}
}_W
\underbrace{
\begin{bmatrix}
 1 & \0 & \0 \\
 \0 & 1 & \0 \\
 \0 & \0 & 0
\end{bmatrix}
}_P
=
\underbrace{
\begin{bmatrix}
 1 & \0 & \0 \\
 \0 & \frac{3}{4} & \frac{\sqrt{3}}{4} \\[3pt]
 \0 & \frac{\sqrt{3}}{4} & \frac{1}{4} \\
\end{bmatrix}
}_Q
\underbrace{
\begin{bmatrix}
 1 & \0 & \0 \\[2pt]
 \0 & \frac{\sqrt{3}}{2} & -\frac{1}{2} \\[3pt]
 \0 & \frac{1}{2} &  \frac{\sqrt{3}}{2}\\
\end{bmatrix}
}_W,
\end{equation*}
in which $W$ is unitary and $P,Q$ are orthogonal projections.
\end{Example}

For each $A \in \M_n$, the positive semidefinite matrix $A^*A$ has a unique
positive semidefinite square root $(A^*A)^{1/2}$, usually denoted $|A|$.  In fact, $|A| = p(A^*A)$
for any polynomial $p$ with the property that $p(\lambda) = \lambda^{1/2}$
for each $\lambda \in \sigma(A^*A) \subseteq [0,\infty)$.

\begin{Theorem}
If $A \in \M_n$, then there is a unique partial isometry $E \in \M_n$ 
and positive semidefinite $R \in \M_n$ 
so that $A = ER$ and $\ker E= \ker R$.  In fact, $R = |A|$.
\end{Theorem}

\begin{proof}
Let $A \in \M_n$ and $r = \rank A$.
Write a singular value decomposition $A = U\Sigma V^*$ and observe that
$A^*A = V\Sigma^2 V^*$ and hence $|A| = V\Sigma V^*$.  Then
$A = E R$, in which $E = U X_r V^*$ is a partial isometry
and $R =|A|$.  Moreover, $\ker E = \ker R$ by construction.
This establishes the existence of the desired factorization.

Now suppose that $A = FS$, in which $F \in \M_n$ is a partial isometry, $S \in \M_n$ is positive semidefinite,
and $\ker F = \ker S$.  Then 
$A^*A = S^*F^*FS = S^*S$
since $F^*F$ is the orthogonal projection onto $(\ker F)^{\perp} = \ran S$.  The uniqueness of the positive semidefinite
square root of a positive semidefinite matrix ensures that $S = |A|$.  In particular, $\ker F = \ker |A| = \ker E$.
Let $\vec{y} \in (\ker F)^{\perp} = \ran |A|$.  Then $\vec{y} = |A| \vec{x}$ for some $\vec{x} \in \C^n$
and hence $F\vec{y} = F|A|\vec{x} = A\vec{x} = E|A|\vec{x} = E \vec{y}$.  Thus, $E = F$.
\end{proof}

\subsection{Products of partial isometries}\label{Section:Products}
The set of partial isometries is not closed under multiplication.  For example,
\begin{equation*}
\minimatrix{0}{1}{0}{0}
\begin{bmatrix}
0 & \frac{1}{\sqrt{2}} \\[4pt]
0 & \frac{1}{\sqrt{2}} 
\end{bmatrix}
= \minimatrix{0}{\frac{1}{\sqrt{2}}}{0}{0}
\end{equation*}
is the product of partial isometries but is not a partial isometry.
The main result of this section (Theorem \ref{Theorem:Product}) is a criterion
for when the product of two partial isometries is a partial isometry.
The proof requires two preparatory lemmas.

\begin{Lemma}\label{Lemma:ContractiveIdempotent}
If $A \in \M_n$ is idempotent and $\norm{A} \leq 1$, then $A$ is an orthogonal projection.
\end{Lemma}

\begin{proof}
Suppose that $A \in \M_n$ is idempotent and $\norm{A} \leq 1$.  For $\vec{x} \in \C^n$,
\begin{align*}
\norm{A \vec{x} - A^*A \vec{x}}^2
&=\norm{ A \vec{x}}^2 + \norm{ A^*A \vec{x}}^2 - 2 \Re\inner{A \vec{x} , A^*A\vec{x}} \\
&\leq\norm{ A \vec{x}}^2 + \norm{A^*}^2\norm{ A \vec{x}}^2 - 2 \Re\inner{A^2 \vec{x} , A\vec{x}} \\
&\leq \norm{A \vec{x}}^2 + \norm{ A\vec{x}}^2 - 2 \Re \inner{A \vec{x}, A\vec{x}} \\
&= \norm{A \vec{x} - A\vec{x}}^2 \\
&=0.
\end{align*}
Thus, $A = A^*A$ is Hermitian and hence $A$ is an orthogonal projection.
\end{proof}

\begin{Lemma}\label{Lemma:ProjectionProduct}
Let $P, Q \in \M_n$ be orthogonal projections. Then $P Q$ is a partial isometry 
if and only if it is an orthogonal projection. 
\end{Lemma}

\begin{proof}
Let $P, Q \in \M_n$ be orthogonal projections.
If $A = PQ$ is a partial isometry, then $\norm{A} = \norm{PQ} \leq \norm{P}\norm{Q} \leq 1$ and
$A  = A A^* A
 = (P Q) (Q P) (P Q)
 = (P Q) (PQ)
 = A^2$,
so Lemma \ref{Lemma:ContractiveIdempotent} ensures that $A$ is an orthogonal projection. 
Conversely, if $A = PQ$ is an orthogonal projection, then it is a partial isometry.
\end{proof}

With the preceding two lemmas, we can prove the following result \cite[Thm.~5]{MR0227194}.

\begin{Theorem}\label{Theorem:Product}
Let $A, B \in \M_{n}$ be partial isometries. Then $A B$ is a partial isometry
if and only if $A^*A$ and $ B B^*$ commute. 
\end{Theorem}

\begin{proof}
Let $A, B \in \M_{n}$ be partial isometries.
Write $A = UP$ and $B = QV$, in which $U,V \in \M_n$ are unitary and $P = A^*A$
and $Q = AA^*$ are orthogonal projections.

\smallskip\noindent($\Rightarrow$)
If $AB \in \M_n$ is a partial isometry,
then $AB = UPQV$ is a partial isometry, so $PQ$ is a partial isometry (Proposition \ref{Proposition:Basic}).
Lemma \ref{Lemma:ProjectionProduct} ensures that $PQ$ is an orthogonal
projection, so $PQ = (PQ)^* = Q^*P^* = QP$.  Thus, $A^*A$ and $BB^*$ commute.

\smallskip\noindent($\Leftarrow$)
If $P = A^*A$ and $Q=BB^*$ commute, then $PQ$ is a partial isometry since
$(P Q) (P Q)^* (P Q) = P Q Q^* P^* P Q
 = P Q P Q
 = P Q$.
Thus, $AB = (UP)(QV) = U(PQ)V$ is a partial isometry.
\end{proof}

\begin{Example}
The partial isometries
\begin{equation*}
A = 
\begin{bmatrix}
 1 & 0 & \0 \\[3pt]
 0 & \frac{1}{2} & \0 \\[3pt]
 0 & \frac{\sqrt{3}}{2} & \0 \\
\end{bmatrix}
\qquad\text{and}\qquad
B=
\begin{bmatrix}
 \0 & \0 & \0 \\[3pt]
 \frac{2}{3} & \frac{2}{3} & \frac{1}{3} \\[3pt]
 \frac{1}{3} & -\frac{2}{3} & \frac{2}{3} \\
\end{bmatrix}
\end{equation*}
satisfy $A^*A = \diag(1,1,0)$ and $BB^* = \diag(0,1,1)$.  Since $A^*A$ and $B^*B$ commute,
Theorem \ref{Theorem:Product} implies that
\begin{equation*}
AB = 
\begin{bmatrix}
 0 & 0 & 0 \\[2pt]
 \frac{1}{3} & \frac{1}{3} & \frac{1}{6} \\[4pt]
 \frac{1}{\sqrt{3}} & \frac{1}{\sqrt{3}} & \frac{1}{2 \sqrt{3}} \\
\end{bmatrix}
\end{equation*}
is a partial isometry (it is a partial isometry of rank one).
\end{Example}

Theorem \ref{Theorem:MP} ensures that $A \in \M_n$ is a partial isometry if and only if $A^* = A^+$.
This yields the following result of Erd\'elyi \cite[Thm.~3]{ErdelyiProduct} (this paper contains
several other results concerning products of partial isometries).

\begin{Proposition}
Let $A_1,A_2,\ldots,A_k \in \M_n$ be partial isometries.
Then $A_1 A_2 \cdots A_k$ is a partial isometry if and only if
$(A_1A_2\cdots A_n)^+ = A_n^+ A_{n-1}^+ \cdots A_1^+$.
\end{Proposition}

Any product of partial isometries is a contraction.  Which contractions are products of partial isometries?
A precise answer was provided by Kuo and Wu \cite{MR977922}.

\begin{Theorem}
For a contraction $A \in \M_n$ the following are equivalent. 
\begin{enumerate}
\item $A$ is the product of $k$ partial isometries.
\item $\rank(I - A^* A) \leq k \dim \ker A$.
\item $|A| = (A^* A)^{1/2}$ is the product of $k$ idempotent matrices. 
\end{enumerate}
\end{Theorem}

Since the proof of the Kuo--Wu theorem is long and somewhat computational, we do not include it here.
Their theorem provides the following interesting corollary.

\begin{Corollary}\hfill
\begin{enumerate}
\item Any contraction $A \in \M_n$ can be factored into a finite product of partial isometries if and only if $A$ is unitary or singular.
\item Any singular contraction can be factored as a product of $n$ partial isometries.
\item There are singular contractions that cannot be factored as a product of $n - 1$ partial isometries.
\end{enumerate}
\end{Corollary}

See Theorem \ref{Theorem:MAMB} for another problem concerning products of partial isometries.

Although the matrix product of two partial isometries need not be a partial isometry, their Kronecker product is.

\begin{Proposition}
Let $A \in \M_m$ and $B \in \M_n$.
Then $A \otimes B$ is a partial isometry if and only if $A$ and $B$ are partial isometries.
\end{Proposition}

\begin{proof}
This follows from the fact that
$AA^*A \otimes BB^*B = (A \otimes B)(A \otimes B)^*(A \otimes B)$; 
see \cite[Sect.~3.6]{GarciaHorn} for properties of the Kronecker product.
\end{proof}

\section{Similarity}\label{Section:Similar}
In this section we consider similarity invariants, such as the spectrum, characteristic polynomial,
and Jordan canonical form, of partial isometries.  Among other things, we discuss a recent result of the first author
and David Sherman, who solved the similarity problem for partially isometric matrices \cite{MSPI}.

\subsection{Spectrum and characteristic polynomial}\label{Section:Spectrum}
In this section we describe the spectrum and characteristic polynomial of a partial isometry.

\begin{Proposition}\label{Proposition:Spectrum}
If $A \in \M_n$ is a partial isometry, then
$\sigma(A) \subset \D^-$.  Moreover,
$0 \in \sigma(A)$ if and only if $A$ is not unitary.
\end{Proposition}

\begin{proof}
If $A \vec{z} = \lambda \vec{z}$ and $\|\vec{z}\|=1$, then 
$|\lambda|  = \|A \vec{z}\| \leq \|A\| \|\vec{z}\| \leq 1$
by Corollary \ref{Corollary:Norm}.  Thus, $\sigma(A) \subset \D^-$.
For the second statement, observe that $0 \notin \sigma(A)$ if and only if
$A$ is an invertible partial isometry, that is, $A$ is unitary.
\end{proof}

Not every finite subset of $\D^-$ is the spectrum of a partial isometry.  Proposition \ref{Proposition:Spectrum}
ensures that $0$ is an eigenvalue of every non-unitary partial isometry.  
Halmos and McLaughlin proved that this is essentially the only restriction \cite[Thm.~3]{Halmos}.  

\begin{Theorem}\label{Theorem:Halmos}
Every monic polynomial whose roots lie in $\D^-$ and include zero
is the characteristic polynomial of a (non-unitary) partial isometry.
\end{Theorem}

\begin{proof}
We proceed by induction on the degree $n$ of the polynomial.
The base case $n=1$ concerns the polynomial $z$, which is the characteristic polynomial of the
$1 \times 1$ partial isometry $[0]$.
For our induction hypothesis, suppose that every monic polynomial of degree $n-1$ whose roots lie in $\D^-$
and include zero is the characteristic polynomial of a partial isometry.
Suppose that $p$ is a polynomial of degree $n$ whose roots lie in $\D^-$ and include $0$.  
There are two possibilities.
\begin{enumerate}[wide, labelwidth=!, labelindent=0pt]
\item If the other $n-1$ roots of $p$ lie on $\T$, then there is a unitary $U \in \M_{n-1}$ with these roots as eigenvalues,
repeated according to multiplicity.  The characteristic polynomial of the partial isometry $U \oplus [0] \in \M_n$ is $p(z)$, as desired.
\item If $p(z)/z$ has a root $\lambda \in \D$, then $p(z) = (z - \lambda)q(z)$, in which $q$ is monic, has zero as a root, and $\deg q = n-1$.
The induction hypothesis give a partial isometry $A$ with characteristic polynomial $q(z)$.
Since $0 \in \sigma(A)$, it follows that $\rank A \leq n-1$
and hence there is a $\vec{z} \in (\ran A)^{\perp}$ with $\| \vec{z} \|^2 = 1 - |\lambda|^2$.  Now verify that
\begin{equation*}
\minimatrix{A}{\vec{z}}{\vec{0}^{\mathsf{T}}} \lambda \in \M_{n+1}
\end{equation*}
is a partial isometry with characteristic polynomial $(z-\lambda)q(z) = p(z)$.
\end{enumerate}
This completes the induction.
\end{proof}

\begin{Example}
$z(z-\frac{1}{2})$ is the characteristic polynomial of the partial isometry
\begin{equation*}
    \begin{bmatrix}
        0 & \frac{ \sqrt{3} }{2} \\[2pt]
        0 & \frac{1}{2}
    \end{bmatrix}   .
\end{equation*}
On the other hand, $(z-\frac{1}{2})^2$ is not the characteristic polynomial of a partial isometry.
If it were, then the partial isometry would be invertible ($0$ is not an eigenvalue) and hence unitary.
Thus, its eigenvalues would lie on $\T$, which is not the case.
\end{Example}

There is another proof, which appeared in \cite{MSPI}, of Theorem \ref{Theorem:Halmos} that is of independent interest
because of its critical use of the Weyl--Horn inequalities \cite{HornAlfred,Weyl}.\footnote{The Horn
in question is Alfred Horn, not the Roger A.~Horn of \emph{Matrix Analysis} fame \cite{HornJohnson}.}

\begin{Theorem}\label{Theorem:WeylHorn}
There is an $n \times n$ matrix with singular values
$\sigma_1 \geq \sigma_2 \geq \cdots \geq\sigma_n \geq 0$ and eigenvalues $\lambda_1,\lambda_2,\ldots,\lambda_n$,
indexed so that $|\lambda_1| \geq |\lambda_2| \geq \cdots \geq |\lambda_n|$, if and only if
\begin{equation*}
\sigma_1 \sigma_2 \cdots \sigma_n = |\lambda_1 \lambda_2 \cdots \lambda_n|
\qquad\text{and}\qquad
    \sigma_1 \sigma_2 \cdots \sigma_k \geq |\lambda_1 \lambda_2 \cdots \lambda_k| 
\end{equation*}
for $k = 1,2,\ldots,n-1$.
\end{Theorem}

Suppose that $\lambda_1,\lambda_2,\ldots,\lambda_n \in \D$ are indexed so that
\begin{equation*}
|\lambda_1| \geq |\lambda_2| \geq \cdots \geq 
|\lambda_r| > |\lambda_{r+1}| = \cdots = |\lambda_n| = 0;
\end{equation*}
that is, the final $n-r$ terms in the sequence are $0$.
If we let
\begin{equation*}
\sigma_1 = \sigma_2 = \cdots = \sigma_r = 1 \qquad \text{and}\qquad  \sigma_{r+1} = \cdots = \sigma_n = 0,
\end{equation*}
then Theorem \ref{Theorem:WeylHorn} provides an $A \in \M_n$ with singular values
$\sigma_1,\sigma_2,\ldots,\sigma_n$ and eigenvalues $\lambda_1,\lambda_2,\ldots,\lambda_n$.
The singular values of $A$ are in $\{0,1\}$, so $A$ is a partial isometry whose characteristic polynomial
has the prescribed roots.

\subsection{Similarity and Jordan form}\label{Section:Jordan}
Theorem \ref{Theorem:Halmos} describes the possible
characteristic polynomials of partial isometries.  The following examples show that this does not settle the similarity problem for the class.

\begin{Example}
The partial isometries
\begin{equation*}
\minimatrix{0}{1}{0}{0}
\qquad \text{and} \qquad
\minimatrix{0}{0}{0}{0}
\end{equation*}
have the same characteristic polynomial, namely $z^2$, but they are not similar
since their ranks differ.  
\end{Example}

A complicating issue is that the property ``similar to a partial isometry'' is
not inherited by direct summands.  Consider the next example.

\begin{Example}
The $1 \times 1$ matrix $[ \frac{1}{2} ]$ is a direct summand of $\diag(0,\frac{1}{2})$, 
which is similar to the partial isometry
\begin{equation*}
    \begin{bmatrix}
        0 & \frac{ \sqrt{3} }{2} \\[3pt]
        0 & \frac{1}{2}
    \end{bmatrix}   .
\end{equation*}
However, $[\frac{1}{2}]$ is not similar to a partial isometry since the spectrum of a non-unitary partial isometry
must include $0$ (Proposition \ref{Proposition:Spectrum}).  
\end{Example}

The following theorem is due to the first author and David Sherman \cite{MSPI}.
The proof requires several lemmas and is deferred until the end of this section.
In what follows, let $J_n(\lambda)$ denote the $n \times n$ Jordan block with eigenvalue $\lambda$. 
Recall that every $n \times n$ matrix is similar to a direct
sum of Jordan blocks \cite[Thm.~11.2.14]{GarciaHorn}.
The nullity of $A - \lambda I$ equals the number of Jordan blocks for the eigenvalue $\lambda$.

    \begin{Theorem}\label{Theorem:SimilarPI}
        $A \in \M_n$ is similar to a partial isometry if and only if the following conditions hold.
        \begin{enumerate}
            \item $\sigma(A) \subseteq \D^-$.
            \item If $\zeta \in \sigma(A) \cap \T$, then its algebraic and geometric multiplicities are equal.
            \item $\dim\ker A \geq \dim\ker (A - \lambda I)$ for each $\lambda \in \sigma(A) \cap \D$.
        \end{enumerate}
    \end{Theorem}

   Condition (b) ensures that the Jordan blocks 
   for each eigenvalue of unit modulus are all $1 \times 1$  and (c) tells us that no eigenvalue in $\D$
   can give rise to more Jordan blocks than $0$ does.  Consequently,
   \begin{equation*}
   	\begin{bmatrix}
	\frac{1}{2} & & & \\
	& \frac{1}{2} & & \\
	&&0&\\
	&&&0\\
	\end{bmatrix},
	\qquad
	   	\begin{bmatrix}
	\frac{1}{2} & 1& & \\
	& \frac{1}{2} & & \\
	&&0&\\
	&&&0\\
	\end{bmatrix}
	\quad\text{and}\quad
   	\begin{bmatrix}
	\frac{1}{2} & 1& & \\
	& \frac{1}{2} & & \\
	&&0&1\\
	&&&0\\
	\end{bmatrix}   
\end{equation*}
are possible Jordan forms for a partial isometry, while
\begin{equation*}
   	\begin{bmatrix}
	\frac{1}{2} & & & \\
	& \frac{1}{2} & & \\
	&&0&1\\
	&&&0\\
	\end{bmatrix}
\end{equation*}
is not.
    
The first lemma that we need is a
variation of Theorem \ref{Theorem:Halmos}.  
To prescribe the Jordan canonical form of the resulting upper-triangular partial isometry, 
we need to control the entries on its first superdiagonal.
    
    \begin{Lemma}\label{Lemma:Superdiagonal}
        For any $\xi_1,\xi_2,\ldots, \xi_{n-1} \in \D$, there
        exists an upper-triangular partial isometry 
        $V  \in \M_n$ such that 
        \begin{enumerate}
            \item the diagonal of $V$ is $(0,\xi_1,\xi_2,\ldots, \xi_{n-1})$,
            \item the final $n-1$ columns of $V$ are are orthonormal, and
            \item each entry of $V$ on the first superdiagonal is nonzero.
        \end{enumerate}
    \end{Lemma}

    \begin{proof}
        We proceed by induction on $n$.  For the base case $n = 2$,
        \begin{equation*}
            V = \minimatrix{ 0}{ \sqrt{1 - | \xi_1|^2} }{0}{ \xi_1}
        \end{equation*}
        is a partial isometry with the desired properties.  For the induction hypothesis, suppose that the lemma holds for some $n$.        
        Suppose that $\xi_1,\xi_2,\ldots, \xi_{n} \in \D$ are given and apply the induction hypothesis to
        $\xi_1,\xi_2,\ldots, \xi_{n-1}$ to obtain an upper-triangular partial isometry $V \in \M_n$ that satisfies (a), (b), and (c).
	Since the first column of $V$ is $\vec{0}$, there is a 
        $\vec{v} \in (\ran V)^{\perp}$ with $\| \vec{v} \| = \sqrt{1 - | \xi_n|^2}$.  Define
        \begin{equation*}
        V' = \minimatrix{V}{\vec{v}}{\vec{0}^{\mathsf{T}}}{\xi_n} \in \M_{n+1}.
        \end{equation*}
        Then $V'$ has diagonal $(0,\xi_1,\xi_2,\ldots, \xi_{n})$.  Its first column is $\vec{0}$ and its final $n$ columns are orthonormal, 
        so $V'$ is a partial isometry.
        The entries of $V$ on the first superdiagonal are nonzero, so all of the entries of $V'$ on its first superdiagonal are
        nonzero, except possibly the $(n,n+1)$ entry.  Suppose toward a contradiction that
        \begin{equation*}
            V' =\left[
            \begin{array}{c|ccccc}
                0 & v_{1,2} & \cdots & v_{1,n-1} & v_{1,n} & v_{1,n+1} \\
                0 & \xi_1 & \cdots & v_{2,n-1} & v_{2,n} & v_{2,n+1} \\[-2pt]
                0 & 0 & \ddots & \vdots & \vdots & \vdots \\
                0 & 0 & \cdots & \xi_{n-1} & v_{n-1,n} & v_{n-1,n+1} \\
                \hline
                0 & 0 & \cdots & 0 & \xi_{n-1} & 0 \\
                0 & 0 & \cdots & 0 & 0 & \xi_n 
            \end{array}
            \right].
        \end{equation*}
        Then the upper right $(n-2) \times (n-1)$ submatrix has $n-1$ orthogonal nonzero columns, which is impossible.
        Thus, each entry on the first superdiagonal of $V'$ is nonzero.  This completes the induction.
    \end{proof}

\begin{Lemma}\label{Lemma:JCFB}
    If $T \in \M_n$ is upper triangular with $\sigma(T) = \{ \lambda\}$,
    and the entries on the first superdiagonal of $T$ are all nonzero, then $T \sim J_n(\lambda)$.
\end{Lemma}

\begin{proof}
    The superdiagonal condition ensures that $\rank (T - \lambda I) = n-1$ since the reduced row echelon form of $T - \lambda I$
    has exactly $n-1$ leading ones.
      Thus, the Jordan canonical form of $T$
    is $J_n(\lambda)$.
\end{proof}

      The following lemma is \cite[Theorem 2.4.6.1]{HornJohnson}:
	
	\begin{Lemma}\label{Lemma:Jordan}
		Suppose that $T = [T_{ij}]_{i,j}^{d}\in \M_n$ is block upper triangular,
		and each $T_{ii} \in \M_{n_{i}}(\C)$ is upper triangular with all diagonal entries equal to $\lambda_i$.
		If $\lambda_i \neq \lambda_j$
		for $i \neq j$, then $T \sim T_{11} \oplus T_{22} \oplus \cdots \oplus T_{dd}$.
	\end{Lemma}
	
We are now ready for the proof of Theorem \ref{Theorem:SimilarPI}.	

\begin{proof}[Proof of Theorem \ref{Theorem:SimilarPI}]
($\Rightarrow$) Since conditions (a), (b), and (c) of Theorem \ref{Theorem:SimilarPI}
are preserved by similarity, it suffices to show that all three conditions are satisfied by any partial isometry.
Conditions (a) and (b) are implied by Proposition \ref{Proposition:Spectrum} and
Theorem \ref{Theorem:CNU}, respectively, so we focus on (c).
Suppose that $A \in \M_n$ is a partial isometry with $\rank A = r$.  Then $A = UP$, in which $U$ is unitary and $P$ is an orthogonal
projection of rank $r$ (Theorem \ref{Theorem:Polar}).  If $\lambda \in \D$, then the unitarity of $U$ ensures that $U - \lambda I$ is invertible and hence
\begin{align*}
    n 
    &= \rank( U - \lambda I)  \\
    &=\rank\big( (UP - \lambda I) + U(I-P) \big)  \\ 
    &\leq \rank(A - \lambda I) + \rank(I-P) \\
    &= \rank(A- \lambda I) + n-r.
\end{align*}
Thus, $\rank A \leq \rank(A - \lambda I)$ from which (c) follows.
\smallskip

\noindent ($\Leftarrow$)
Suppose that $A \in \M_n$ satisfies conditions (a), (b), and (c) of Theorem \ref{Theorem:SimilarPI}.
Condition (b) ensures that
$A \sim A' \oplus U$, in which $\sigma(A') \subset \D$ and 
$U$ is a diagonal matrix whose eigenvalues are on $\T$ (either summand may be vacuous).
Then $U$ is unitary, so $A$ is similar to a partial isometry if and only if $A'$ is.

Without loss of generality, we may assume that $\sigma(A) \subset \D$.
Proposition \ref{Proposition:Spectrum} ensures that $0 \in \sigma(A)$.
Then $m = \nullity A$ is the number of Jordan blocks for the eigenvalue $0$
in the Jordan canonical form of $A$.  Moreover, condition (c) implies that the Jordan canonical form
of $A$ has at most $m$ Jordan blocks for any nonzero eigenvalue of $A$.
Thus, it suffices to show that any matrix of the form
\begin{equation}\label{eq:ASDSMFB}
B = J_{n_0}(0) \oplus \bigoplus_{i=1}^d J_{n_i}(\lambda_i),\qquad d \geq 0, \quad 0 < n_i \leq n,
\end{equation}
in which 
$\lambda_1,\lambda_2,\ldots, \lambda_d \in \D \backslash\{0\}$ are distinct, is similar to a partial
isometry.  This is because $A$ is similar to a direct sum of matrices of the form \eqref{eq:ASDSMFB}.

Lemma \ref{Lemma:Superdiagonal} ensures that there exists
a partial isometry $V \in \M_n$ 
with nonzero entries on its first superdiagonal and
whose diagonal entries are
\begin{equation*}
\underbrace{ 0,0,\ldots,0}_{\text{$n_0$ times}}, \underbrace{ \lambda_1,\lambda_1,\ldots,\lambda_1}_{\text{$n_1$ times}}, 
\underbrace{\lambda_2,\lambda_2,\ldots,\lambda_2}_{\text{$n_2$ times}},\ldots,
\underbrace{\lambda_d,\lambda_d,\ldots,\lambda_d}_{\text{$n_d$ times}},
\end{equation*}
in that order.  
Partition $V$ with respect to $B$, that is, so that $V_{i,j} \in \M_{n_i \times n_j}$.
Then Lemma \ref{Lemma:JCFB} implies that $V_{i,i} \sim J_{n_i}(\lambda_i)$ and hence
$V \sim B$ by Lemma \ref{Lemma:Jordan}.  
\end{proof}

\section{Unitary similarity}\label{Section:UnitarySimilar}
In this section we consider several questions
connected to partial isometries and unitary similarity.  Recall that $A,B \in \M_n$ are unitarily similar if
$A = UBU^*$ for some unitary $U \in \M_n$.  This relationship is denoted $A \cong B$.

\subsection{Partial isometric extension of a contraction}\label{Section:ExtensionPI}
Suppose that $A\in \M_n$ is a contraction, that is, $\norm{A} \leq 1$.
Then $I - A^*A$ is positive semidefinite and has a unique positive semidefinite
square root, denoted $(I - A^*A)^{1/2}$.  We follow Halmos and McLaughlin \cite{Halmos} and define
\begin{equation*}
M(A) = \minimatrix{A}{(I - AA^*)^{1/2}}{0}{0} \in \M_{2n},
\end{equation*}
which is a partial isometry since
\begin{equation}\label{eq:MAMA}
M(A)M(A)^* = \minimatrix{A}{(I - AA^*)^{1/2}}{0}{0}\minimatrix{A^*}{0}{(I - AA^*)^{1/2}}{0}
= \minimatrix{I}{0}{0}{0}
\end{equation}
is an orthogonal projection.  Thus, every contraction is the restriction of a partial isometry
to an invariant subspace.
The matrix $M(A)$ is relevant to the unitary similarity problem for contractions.

 
 

\begin{Theorem}\label{Theorem:M(A)}
Let $A, B \in \M_n$ be contractions.  Then $A \cong B$ if and only if $M(A) \cong M(B)$.
\end{Theorem}

\begin{proof}
Let $A,B \in \M_n$ be contractions.

\medskip\noindent($\Rightarrow$) Suppose that $A \cong B$.
Then $UA = BU$ for some unitary $U \in \M_n$ and hence
$U(AA^*) = (BB^*)U$.  In particular, $AA^*$ and $BB^*$
have the same eigenvalues, all of them nonnegative, with the same multiplicities.
If $p$ is a polynomial such that $p(\lambda) = (1 - \lambda)^{1/2}$ for 
each such eigenvalue, then
\begin{equation*}
p(AA^*) = (I - AA^*)^{1/2} \qquad \text{and} \qquad
p(BB^*) = (I - BB^*)^{1/2}.  
\end{equation*}
Thus,
\begin{equation*}
U(I - AA^*)^{\frac{1}{2}} = Up(AA^*) =p(UAA^*) =  p(BB^*U) = p(BB^*)U = (I - BB^*)^{\frac{1}{2}}U.
\end{equation*}
A computation then confirms that $(U \oplus U)M(A) = M(B)(U \oplus U)$.

\medskip\noindent($\Leftarrow$) 
If $M(A) \cong M(B)$, then \eqref{eq:MAMA} ensures that
\begin{equation*}
A \oplus 0_n = M(A)^2 M(A)^* \cong M(B)^2M(B)^* = B \oplus 0_n.
\end{equation*}
For any word $w(x,y)$ in two noncommuting variables,
\begin{align*}
\tr w(A,A^*) 
&= \tr w\big( A\oplus 0_n, (A\oplus 0_n)^*\big) \\
&= \tr w\big( B\oplus 0_n, (B\oplus 0_n)^*\big) \\
&= \tr w(B,B^*).
\end{align*}
A well-known theorem of Specht ensures that $A \cong B$ \cite{Specht}.\footnote{Pearcy showed that it suffices to consider 
words of total degree $2n^2$ \cite{Pearcy2}.  For $n=3$ and $n=4$, much better results are known \cite{Djokovic, Pearcy, Sibirskii}; see Section \ref{Section:LowDimensions}.}
\end{proof}

\begin{Remark}
The proof that $M(A) \cong M(B)$ implies $A \cong B$ provided in Theorem \ref{Theorem:M(A)}
is inherently finite dimensional because of its reliance on Specht's theorem \cite{Specht}.  This is to be expected,
since the result fails for operators on infinite-dimensional Hilbert spaces:  let $A = I$ and $B = I \oplus 0$.
Then $M(A) \cong I \oplus 0 \cong M(B)$, but $A$ and $B$ are not unitarily similar.
\end{Remark}

The forward implication of Theorem \ref{Theorem:M(A)} is due to Halmos and McLaughlin \cite[Thm.~1]{Halmos}.
They proved the converse under the assumption that $A$ or $B$ is invertible.  A similar method 
applies if $A$ or $B$ is a strict contraction.  Here is the argument.  
Let $M(A) \cong M(B)$ and
suppose without loss of generality that $A$ is invertible or a strict contraction.
Then $UM(A) = M(B)U$ for some unitary
\begin{equation*}
U = \minimatrix{X}{Y}{Z}{W} \in \M_{2n},
\end{equation*}
in which $X,Y,Z,W \in \M_n$.  Thus,
\begin{equation}\label{eq:XAZ}
\minimatrix{XA}{X(I - AA^*)^{\frac{1}{2}}}{ZA}{Z(I - AA^*)^{\frac{1}{2}}}
 =\minimatrix{BX+(I - BB^*)^{\frac{1}{2}}Z}{BY+(I - BB^*)^{\frac{1}{2}}W}{0}{0}.
\end{equation}
If $A$ is invertible, we see from the $(2,1)$ entry above
that $Z = 0$.  If $A$ is a strict contraction, then $(I - AA^*)^{1/2}$ is invertible
and we see from the $(2,2)$ entry in \eqref{eq:XAZ} that $Z = 0$.
However,
\begin{equation}\label{eq:UUXYZW}
\minimatrix{I}{0}{0}{I} = 
I = U^*U = \minimatrix{X^*}{0}{Y^*}{W^*} \minimatrix{X}{Y}{0}{W} 
=\minimatrix{X^*X}{X^*Y}{Y^*X}{Y^*Y + W^*W}
\end{equation}
and hence $X^*X = I$, that is, $X$ is unitary.  
Since $Z = 0$ in \eqref{eq:XAZ}, we see that $XA = BX$, so $A \cong B$.

A related result about products of matrices is due to Erd\'elyi \cite[Thm.~5]{ErdelyiProduct}.

\begin{Theorem}\label{Theorem:MAMB}
Let $A,B \in \M_n$.   Then $M(A)M(B)$ is a partial isometry if and only if $A$ is a partial isometry.
\end{Theorem}

\begin{proof}
Let $A,B \in \M_n$ and define
\begin{equation*}
M =  M(A) M(B) 
= \minimatrix{AB}{A(I - BB^*)^{\frac{1}{2}}}{0}{0}.
\end{equation*}
Then
\begin{equation*}
MM^*
= \minimatrix{AB}{A(I - BB^*)^{\frac{1}{2}}}{0}{0}\minimatrix{B^*A^*}{0}{(I - BB^*)^{\frac{1}{2}}A^*}{0}
= \minimatrix{AA^*}{0}{0}{0}
\end{equation*}
is an orthogonal projection if and only if $AA^*$ is an orthogonal projection.
Thus, $M$ is a partial isometry if and only if $A$ is (Proposition \ref{Proposition:Basic}).
\end{proof}

\subsection{Unitary and completely non-unitary parts}\label{Section:CNU}
The spectrum of a partial isometry is contained in $\D^-$ (Proposition \ref{Proposition:Spectrum}).
The following theorem concerns a useful decomposition of a partial isometry $A \in \M_n$ that corresponds to the partition
$\sigma(A) = (\sigma(A) \cap \D) \cup (\sigma(A) \cap \T)$ of its spectrum.

\begin{Theorem}\label{Theorem:CNU}
Let $A \in \M_n$ be a partial isometry.  Then $A \cong T \oplus U$, in which 
$T$ is an upper-triangular partial isometry with $\sigma(T) \subset \D$ and $U$ is unitary 
(either summand may be absent).
\end{Theorem}

\begin{proof}
If $A \in \M_n$ is a partial isometry, then Schur's theorem on unitary triangularization 
implies that
\begin{equation*}
A \cong \minimatrix{T}{B}{0}{U},
\end{equation*}
in which $\sigma(T) \subset \D$ and $U$ is upper-triangular with $\sigma(U) \subset \T$ \cite[Thm.~10.1.1]{GarciaHorn}.  
Since $A$ is a contraction, each of its columns has norm at most $1$.  Since every entry on the main diagonal of $U$
has unit modulus, it follows that $U$ is diagonal and hence $B = 0$.  Thus, $A \cong T \oplus U$, in which 
$U$ is unitary and $\sigma(T) \subset \D$.  Since $AA^*A = A$, we conclude that $TT^*T = T$, so $T$ is a partial isometry.
\end{proof}

The summand $U$ in Theorem \ref{Theorem:CNU} is the \emph{unitary part} of $A$ and the summand $T$ is the \emph{completely non-unitary} (cnu)
part of $A$.  The latter name arises from the fact that there is no reducing subspace upon which $T$ acts unitarily.
Indeed, otherwise $T \cong T' \oplus U'$, in which $U'$ is unitary (and hence $\sigma(U') \subset \T$), 
and this violates the hypothesis that $\sigma(T) \subset \D$.  In particular, a partial isometry is completely non-unitary if and
only if its spectrum lies in $\D$.

\begin{Example}
The partial isometry
\begin{equation*}
\begin{bmatrix}
0 & \frac{\sqrt{3}}{2} & \0 & \0\\[3pt]
0 & \frac{1}{2} & \0 & \0\\
\0 &  \0& -1& 0 \\
 \0&\0  &0 &1\\
\end{bmatrix}
=
\underbrace{
\begin{bmatrix}
0 & \frac{\sqrt{3}}{2} \\[3pt]
0 & \frac{1}{2} \\
\end{bmatrix}
}_T
 \oplus \underbrace{\minimatrix{1}{0}{0}{-1}}_U
\end{equation*}
is a direct sum of a completely non-unitary partial isometry $N$
and a unitary $U$.
\end{Example}

\begin{Corollary}
Let $A \in \M_n$ be a partial isometry.  
Then $A^n \to 0$ if and only if $A$ is completely non-unitary.
\end{Corollary}

\begin{proof}
One can use the Jordan canonical form of a matrix to show that $A^n \to 0$ if and only if $\sigma(A) \subset \D$
\cite[Thm.~11.6.6]{GarciaHorn}.
\end{proof}

\subsection{Low dimensions}\label{Section:LowDimensions}

In low dimensions there are simple conditions 
to determine when two partial isometries are unitarily similar.
Although the two-dimensional situation is rather straightforward, we include
it for completeness because it suggests a similar approach in dimensions three and four.
Recall that $p_A(z) = \det(zI -A)$ is the characteristic polynomial of $A \in \M_n$.

\begin{Theorem}\label{Theorem:UE2}
Partial isometries $A,B \in \M_2$ are unitarily similar if and only if
$p_A = p_B$ and $\rank A = \rank B$.
\end{Theorem}
 
\begin{proof}
$A,B \in \M_2$ are unitarily similar if and only if $\Phi(A) = \Phi(B)$, in which
$\Phi(x) = (\tr x, \tr x^2, \tr x^*x)$ \cite{Murnaghan}.
Since the trace of a matrix is the sum of its eigenvalues, counted with multiplicity,
$\Phi(A) = \Phi(B)$ occurs for partial isometries $A,B \in \M_n$ if and only if $p_A=p_B$ and $\rank A = \tr A^*A = \tr B^*B = \rank B$.
\end{proof}

The $3 \times 3$ case is slightly more complicated and involves a lemma of
Sibirski{\u\i} \cite{Sibirskii} that
 streamlines an earlier result of Pearcy \cite{Pearcy}.

\begin{Lemma}
$A,B \in \M_3$ are unitarily similar if and only if $\Phi(A) = \Phi(B)$,
in which $\Phi:M_3(\C)\to \C^7$ is 
\begin{equation}\label{eq:Sibirksi}
	\Phi(x) = (\tr x, \, \tr  x^2,\, \tr x^3,\, \tr x^* x,\, \tr x^*x^2, \, \tr {x^*}^2 x^2, \, \tr x^* x^2 {x^*}^2x).
\end{equation}
\end{Lemma}

If $A$ is a partial isometry, then $P_{A^*} = A^*A$ and $P_A = AA^*$ are the orthogonal projections
onto the initial and final spaces of $A$, respectively.  That is,
$\ran P_A = \ran A$ and $\ran P_{A^*} = \ran A^*$.  To extend
Theorem \ref{Theorem:UE2} to $3 \times 3$ matrices, we require an additional condition
concerning such projections.

\begin{Theorem}
Partial isometries $A,B \in \M_3$ are unitarily similar if and only if
\begin{enumerate}
\item $p_A = p_B$,
\item $\rank A = \rank B$, and
\item $\tr P_A P_{A^*} = \tr P_B P_{B^*}$.
\end{enumerate}
\end{Theorem}
	
\begin{proof}
Suppose that $x \in \M_3$ is a partial isometry.
Observe that $(\tr x, \tr x^2, \tr x^3)$ is uniquely determined by the characteristic polynomial of $x$
and that $\rank x = \tr x^*x$.  The invariance of the trace under cyclic permutations of its argument
ensures that 
\begin{align*}
\tr x^*x^2 &= \tr xx^*x = \tr x, \\
\tr {x^*}^2 x^2 &= \tr (x^*x)(xx^*), \quad\text{and}\\
\tr x^* x^2 {x^*}^2x &= \tr (x^*x)(xx^*)(x^*x) = \tr (x^*x)^2(xx^*) = \tr (x^*x)(xx^*).
\end{align*}
Thus, partial isometries $A,B \in \M_3$ are unitarily similar if and only if
conditions (a), (b), and (c) holds.
\end{proof}

In 2007, Djokovi\'c extended the 
Pearcy-Sibirski{\u\i} trace conditions to four dimensions and obtained a complete unitary invariant \cite[Thm.~4.4]{Djokovic}.

\begin{Lemma}
$A,B \in \M_4$ are unitarily similar if and only if 
\begin{equation*}
\tr w_i(A,A^*) = \tr w_i(B,B^*)
\end{equation*}
for $i = 1,2,\ldots,20$, in which the words $w_i(x,y)$ are 
	\begin{multicols}{4}\renewcommand{\labelenumi}{(\arabic{enumi})}
		\begin{enumerate}
			\item $x$
			\item $x^2$
			\item $xy$
			\item $x^3$
			\item $x^2 y$
			\item $x^4$
			\item $x^3 y $
			\item $x^2 y^2$
			\item $xyxy$
			\item $x^3 y^2$
			\item $x^2yx^2y$
			\item $x^2 y^2 xy$
			\item $y^2 x^2 y x$
			\item $x^3 y^2 xy$
			\item $x^3 y^2 x^2 y$
			\item $x^3 y^3 xy$
			\item $y^3 x^3 y x$
			\item $x^3 y x^2 yxy$
			\item $x^2 y^2 x y x^2 y$
			\item $x^3 y^3 x^2 y^2$.
		\end{enumerate}
	\end{multicols}
\end{Lemma}

Using with Djokovi\'c's result, we obtain a complete unitary invariant for $4 \times 4$ partial isometries.

\begin{Theorem}
Partial isometries $A,B \in \M_4$ are unitarily similar if and only if
\begin{enumerate}
\item $p_A = p_B$,
\item $\rank A = \rank B$, 
\item $\tr w(A,A^*) = \tr w(B,B^*)$ for the six words $w(x,y)$ given by
\begin{equation}\label{eq:FinalList}
x^2y^2, \quad x^3y^2 , \quad  x^4y^2, \quad x^3y^3 , \quad x^4y, \quad x^3y^3x^2y^2.
\end{equation}
\end{enumerate}
\end{Theorem}

\begin{proof}
Suppose that $A,B \in \M_4$ are partial isometries, $p_A = p_B$, and $\rank A = \rank B$.
Then $\tr A^k = \tr B^k$ for $k=1,2,3,4$ and hence we need not check words $w_1(x,y) = x$, $w_2(x,y) = x^2$,
$w_4(x,y) = x^3$, and $w_6(x,y) = x^4$ on Djokovi\'c's list.
Since $\tr A^*A = \rank A = \rank B = \tr B^*B$, we can also ignore $w_3(x,y) = xy$.  More words can be proved
redundant when we add the relations $xyx = x$ and $yxy=y$:
\begin{align*}
w_9(x,y) &=xyxy = xy = w_3(x,y), \\
w_{11}(x,y) &=  x(xyx)xy=x^3y = w_7(x,y), \\
w_{12}(x,y) &=x^2y^2xy = x^2 y(yxy) = x^2y^2 = w_8(x,y), \quad\text{and}\\
w_{14}(x,y) &=x^3y^2xy = x^3y(yxy) = x^3y^2 = w_{10}(x,y).
\end{align*}
The invariance of the trace under cyclic permutations yields
\begin{align*}
\tr w_5(x,y) &=\tr x^2y = \tr xyx = \tr x = \tr w_1(x,y),\\
\tr w_7(x,y) &=\tr x^3y = \tr x(xyx) = \tr x^2 = \tr w_2(x,y), \\
\tr w_{13}(x,y) &= \tr y^2x^2yx = \tr x^2(yxy)y = \tr x^2y^2 = \tr w_8(x,y),\\
\tr w_{17}(x,y) &=\tr y^3x^3yx = \tr x^3(yxy)y^2 =  \tr x^3y^3 = \tr w_{16}(x,y),  \quad\text{and}\\
\tr w_{19}(x,y) &= \tr x^2y^2xyx^2y = \tr xy(yxy)x(xyx) = \tr xy^2x^2 = \tr x^3 y^2 = \tr w_{10}(x,y).
\end{align*}
Thus, $w_i(x,y)$ is redundant for $i=1,2,3,4,5,6,7,9,11,12,13,14,17,19$ and we need only consider
the six words $w_i(x,y)$ for $i=8,10,15,16,18,20$.  For some of these, we have simplifications:
\begin{align*}
\tr w_{15}(x,y) &=\tr x^3y^2x^2y =  \tr x^2 y^2 x(xyx) = \tr x^2 y^2 x^2 = \tr x^4 y^2,\\
\tr w_{16}(x,y) &=\tr x^3y^3xy = \tr x^3 y^2(yxy) =  \tr x^3y^3,\\
\tr w_{18}(x,y) &=\tr x^3yx^2yxy = \tr x^2(xyx)x(yxy) = \tr x^4y.
\end{align*}
This yields the list \eqref{eq:FinalList}.
\end{proof}

The words in \eqref{eq:FinalList} can be interpreted more concretely in terms of orthogonal projections.  For example,
if $A \in \M_4$ is a partial isometry, then the trace corresponding to the second word in \eqref{eq:FinalList} is
$\tr A^3 {A^*}^2 = \tr A(AA^*)(A^*A) = \tr P_{A^*} A P_A$.  The interested reader might pursue this further.

\color{black}
\subsection{Defect index one}\label{Section:UnitarySimilarity}
Let $A \in \M_n$ be a partial isometry.  Then 
$\dim \ker A$
is the \emph{defect index} of $A$.  It measures, in a crude sense, the extent to which
$A$ fails to be unitary.  Indeed, if $A$ is unitary, then its defect index is $0$.

The following theorem of Halmos and McLaughlin provides a criterion 
to determine whether two partial isometries with defect index one are unitarily similar
\cite{Halmos}.

\begin{Theorem}\label{Theorem:HalmosUS}
Let $A,B \in \M_n$ be partial isometries with one-dimensional kernels.
Then $A \cong B$ if and only if they have the same characteristic polynomial.
\end{Theorem}

\begin{proof}
\noindent$(\Rightarrow)$ 
If $A,B \in \M_n$ are partial isometries and 
$A \cong B$, then $A$ and $B$ have the same characteristic polynomials \cite[Thm.~9.3.1]{GarciaHorn}.

\smallskip\noindent$(\Leftarrow)$ 
We proceed by induction on $n$.
The base case $n=1$ is true because every
$1 \times 1$ partial isometry with one-dimensional kernel is $[0]$.
Suppose for our induction hypothesis that two $n \times n$ partial isometries with defect index one are unitarily similar
whenever they have the same characteristic polynomial.

Let $A,B \in \M_{n+1}$ be partial isometries with defect index one and suppose that $p_A= p_B$.  Note that $0$ is an eigenvalue of both $A$ and $B$.
In light of Schur's theorem on unitary triangularization \cite[Thm.~10.1.1]{GarciaHorn}, we may assume that 
\begin{equation*}
A = 
\begin{bmatrix}
 A' & \vec{a} \\
 \vec{0}^* & \alpha \\
\end{bmatrix} 
\qquad\text{and}\qquad B = 
\begin{bmatrix}
 B' & \vec{b} \\
 \vec{0}^* &  \alpha \\
\end{bmatrix},
\end{equation*}
in which $\alpha \in \C$, $\vec{a}, \vec{b} \in \C^n$, and 
$A', B' \in \M_n$ are upper-triangular matrices with
$p_{A'} = p_{B'}$ and $0$ as their $(1,1)$ entries.
Because $A$ and $B$ have one-dimensional kernels and $A',B'$ have first column $\vec{0}$, 
we have $\ker A = \ker B = \vecspan\{\vec{e}_1\}$ and
\begin{equation}\label{eq:ABXY}
A = [\vec{0}\,\,X]
\qquad \text{and} \qquad
B = [\vec{0}\,\,Y],
\end{equation}
in which $X,Y \in \M_{(n+1) \times n}$.  Use \eqref{eq:ABXY} to compute
$A^*A = B^*B$, the orthogonal projection onto
$\{\vec{e}_1\}^{\perp}$, and deduce that
\begin{equation*}
X^*X = Y^*Y = I_n
\end{equation*}
and $\norm{\vec{a}}= \norm{\vec{b}}$.  Then $X$
has orthonormal columns and hence the final $n$ columns of $A$ are orthonormal.
Consequently, the final $n-1$ columns of $A'$ are orthonormal and its first column is $\vec{0}$.
This implies that $A'$ is a partial isometry with one-dimensional kernel.  The same reasoning
applies to $B'$.  Since $p_{A'} = p_{B'}$, the induction hypothesis provides
a unitary $W' \in \M_n$ such that 
$W'A' = B'W'$.
Let $\xi \in \T$ be such that $\xi W' \vec{a} = \vec{b}$.
Then $W = W'\oplus [\xi] \in \M_{n+1}$ is unitary and $WAW^* = B$ since
\begin{equation*}
\minimatrix{W}{\vec{0}}{\vec{0}^*}{\xi}
\minimatrix{A'}{\vec{a}}{\vec{0}^*}{\alpha}
\minimatrix{W^*}{\vec{0}}{\vec{0}^*}{\overline{\xi}}
= \minimatrix{WA'W^*}{\overline{\xi}W\vec{a}}{\vec{0}^*}{\alpha}
= \minimatrix{B'}{\vec{b}}{\vec{0}^*}{\alpha}.
\end{equation*}
This completes the induction.
\end{proof}

\begin{Example}\label{Example:HalmosBad}
The matrices
\begin{equation}\label{ABHML}
A = 
\begin{bmatrix}
 0 & 1 & 0 & \0 \\
 0 & 0 & 1 & \0 \\
 0 & 0 & 0 & \0 \\
 \0 & \0 & \0 & 0 \\
\end{bmatrix}
\qquad \text{and} \qquad
B=
\begin{bmatrix}
 0 & 1 & \0 & \0 \\
 0 & 0 & \0 & \0 \\
 \0 & \0 & 0 & 1 \\
 \0 & \0 & 0 & 0 \\
\end{bmatrix}
\end{equation}
are partial isometries, $p_A = p_B$, and $\dim \ker A = \dim \ker B = 2$.
However, $A$ and $B$ are not similar since they
have different Jordan canonical forms.  In particular, $A$ and $B$ are not unitarily similar.
Thus, the one-dimensional kernel condition in Theorem \ref{Theorem:HalmosUS} cannot be ignored.
We will see another unitary invariant that rectifies this in Section \ref{Section:Livsic}.
\end{Example}

\subsection{The transpose of a partial isometry}\label{Section:Transpose}
Although every $A \in \M_n$ is similar to $A^{\mathsf{T}}$ \cite[Thm.~11.8.1]{GarciaHorn},
it is not always the case that $A \cong A^{\mathsf{T}}$ \cite[Pr.~159]{HalmosLA}.  In this section
we tackle the question of when $A \cong A^{\mathsf{T}}$ for a partial isometry $A \in \M_n$.

\begin{Proposition}\label{Proposition:Transpose}
If $A \in \M_n$ is a partial isometry with one-dimensional kernel, then $A \cong A^{\mathsf{T}}$.
\end{Proposition}

\begin{proof}
Suppose that $A \in \M_n$ is a partial isometry with one-dimensional kernel.
Then $A^{\mathsf{T}}$ is a partial isometry with one-dimensional kernel and $p_A = p_{A^{\mathsf{T}}}$,
so Theorem \ref{Theorem:HalmosUS} ensures that $A \cong A^{\mathsf{T}}$.
\end{proof}

Example \ref{Example:UET} below demonstrates that a partial isometry with
two-dimensional kernel need not be unitarily similar to its transpose.  Our next lemma
provides a simple condition that ensures a matrix is unitarily similar to its transpose.

\begin{Lemma}\label{Lemma:UET}
If $A \in \M_n$ is unitarily similar to a complex symmetric (self-transpose) matrix, then $A \cong  A^{\mathsf{T}}$.
\end{Lemma}

\begin{proof}
If $S = UAU^*$, in which $S = S^{\mathsf{T}}$ and $U$ is unitary, then
$UAU^* = S = S^{\mathsf{T}} = \overline{U} A^{\mathsf{T}} U^{\mathsf{T}}$
and hence $VA = A^{\mathsf{T}}V$, in which $V = U^{\mathsf{T}}U$ is unitary.
\end{proof}

The condition in Lemma \ref{Lemma:UET}
is sufficient but not necessary.  The first author and James Tener showed that in dimensions eight and above
$A \cong A^{\mathsf{T}}$ may hold while $A$ is not unitarily similar to a complex
symmetric matrix \cite{UET}.  On the other hand, if $A \in \M_n$ for some $n\leq 7$ and $A \cong A^{\mathsf{T}}$,
then $A$ is unitarily similar to a complex symmetric matrix.

The following theorem, whose proof depends upon the theory of
complex symmetric operators \cite{CSO, CSO2, SNCSO}, 
is due to the first author and Warren Wogen \cite{CSPI}.

\begin{Theorem}\label{Theorem:CSO}
A partial isometry
\begin{equation*}
\minimatrix{X}{0}{Y}{0},
\end{equation*}
in which $X$ is square and $X^*X + Y^*Y = I$, 
is unitarily similar to a complex symmetric matrix
if and only if $X$ is.
\end{Theorem}

Theorem \ref{Theorem:N} ensures that any partial isometry is unitarily similar to one of the form
in Theorem \ref{Theorem:CSO}.  Thus, a partial isometry is unitarily similar to a complex symmetric matrix if and only if its restriction to its initial space
has that property.

\begin{Proposition}\label{Proposition:PI14}
If $A \in \M_n$ is a partial isometry and $1 \leq n \leq 4$, then $A \cong A^{\mathsf{T}}$.
\end{Proposition}

\begin{proof}
For $n=1$ the result is obvious.
If $A \in \M_2$ is a partial isometry, it is either $0$, unitary, or has a one-dimensional kernel.
In all three cases, $A \cong A^{\mathsf{T}}$.  An alternate approach is to use Lemma \ref{Lemma:UET}
after noting that every
$2\times 2$ matrix is unitarily similar to a complex symmetric matrix \cite[Cor.~1]{SNCSO}.

Suppose that $A \in \M_3$ is a partial isometry.  If $\rank A = 0$, then $A = 0$ and we are done.
If $\rank A = 1$, then $A$ is unitarily similar to a complex symmetric matrix \cite[Cor.~5]{SNCSO} and 
we may apply Lemma \ref{Lemma:UET}.
If $\rank A = 2$, then $A$ has a one-dimensional kernel and hence Proposition \ref{Proposition:Transpose} implies that $A \cong A^{\mathsf{T}}$.
If $\rank A = 3$, then $A$ is unitary and therefore $A \cong A^{\mathsf{T}}$.

Suppose that $A \in \M_4$ is a partial isometry.  Proceeding as before leaves
only the case $\rank A = 2$ unsettled.  Then $A$ is unitarily similar to 
\begin{equation*}
\minimatrix{B}{0}{C}{0},
\end{equation*}
in which $B,C \in \M_2$ and $B^*B+C^*C=I$, by Theorem \ref{Theorem:N}.  
Since every $2 \times 2$ matrix is unitarily similar to a complex symmetric matrix \cite[Cor.~1]{SNCSO},
Theorem \ref{Theorem:CSO} ensures that $A$ is unitarily similar to a complex symmetric.  Thus, $A \cong A^{\mathsf{T}}$.
\end{proof}

\begin{Example}\label{Example:UET}
The conditions in Propositions \ref{Proposition:Transpose} and 
\ref{Proposition:PI14} are best possible. The partial isometry 
\begin{equation*}
A = 
\begin{bmatrix}
 0 & 1 & 0 & \0 & \0 \\
 0 & 0 & \frac{1}{2} & \0 & \0 \\
 0 & 0 & 0 & \0 & \0 \\
 1 & 0 & 0 & \0 & \0 \\
 0 & 0 & \frac{\sqrt{3}}{2} & \0 & \0 \\
\end{bmatrix}
\end{equation*}
is $5 \times 5$ and has a two-dimensional kernel.
Although $A$ and $A^{\mathsf{T}}$ are similar, they are not unitarily similar since the
unitarily invariant function
$f(x) = \tr x^3 {x^*}^3x^2 {x^*}^2$ assumes the values
$f(A) = \frac{1}{4}$ and $f(A^{\mathsf{T}}) = \frac{1}{16}$.  
The peculiar choice of trace is motivated by
Djokovi\'c \cite[Thm.~4.4]{Djokovic} (see Section \ref{Section:LowDimensions}).
In fact, $f(x)$ is Djokovi\'c's twentieth unitary invariant; the first nineteen are unable to distinguish $A$ from $A^{\mathsf{T}}$.
\end{Example}

\subsection{Liv\v{s}ic characteristic functions}\label{Section:Livsic}
Theorem \ref{Theorem:HalmosUS} provides a simple criterion to determine whether two partial
isometries of defect one are unitarily similar.  Example \ref{Example:HalmosBad} illustrates that
the defect-one condition cannot be overlooked.  A suitable replacement of Theorem 
\ref{Theorem:HalmosUS} for higher defect is due to Liv\v{s}ic \cite{Livsic}.

Let $A \in \M_n$ be a partial isometry with defect $r\geq 1$ whose spectrum contained in $\D$ (so $A$ is completely non-unitary; see Section \ref{Section:CNU}).
Let $\vec{v}_1,\vec{v}_2, \ldots, \vec{v}_r$ be an orthonormal basis for $\ker A$. 
Theorem \ref{Theorem:Polar} ensures that $A$ has a unitary extension $U$ (in fact many of them). 
For $z\in \D$, define
\begin{equation}\label{Liv-Def}
w_{A}(z)= z \big[\langle (U - z I)^{-1} \vec{v}_{i}, \vec{v}_{j}\rangle \big] \big[\langle (U - z I)^{-1} U\vec{v}_{i}, \vec{v}_{j}\rangle \big]^{-1} \in \M_r.
\end{equation}
Liv\v{s}ic showed that $w_{A}$ is an analytic, contractive $\M_r$-valued function on $\D$ such that $w_A(\zeta)$ is unitary for $\zeta \in \T$.
He showed that different choices of 
$\{\vec{v}_1,\vec{v}_2, \ldots, \vec{v}_r\}$ and $U$ result in 
$Q_{1} w_{A} Q_{2}$, where $Q_1, Q_2$ are constant unitary matrices. The function $w_{A}$ (more precisely, the family of functions) is 
the \emph{Liv\v{s}ic characteristic function} of $A$ and
it is a unitary invariant for the partial isometries with defect $r$ \cite{Livsic}.

\begin{Theorem}\label{LivsicThm}
Let $A, B \in \M_n$ be partial isometries with defect $r\geq 1$ and whose spectra are contained in $\D$.
Then $A$ and $B$ are unitarily similar if and only if there are unitary $Q_1,Q_2 \in \M_r$ such that
\begin{equation*}
w_{A}(z) = Q_1 w_{B}(z) Q_2
\end{equation*}
for all $z \in \D$.
\end{Theorem}

\begin{proof}
The details are technical so we only sketch the proof in the case $r=1$.  
Let $A, B \in \M_n$ be partial isometries with defect $1$ and whose spectra are contained in $\D$.
In this case,
$\ker A = \vecspan\{\vec{v}\}$ for some unit vector $\vec{v}$.  Then
\begin{equation*}
w_{A}  = \frac{z \langle (U - z I)^{-1} \vec{v}, \vec{v}\rangle}{\langle (U - z I)^{-1} U \vec{v}, \vec{v}\rangle},
\end{equation*}
in which $U \in \M_n$ is a unitary extension of $A$.  One can verify that $w_A$ only changes by a unimodular constant if one selects
a different $U$.

If $B = V A V^{*}$ for some unitary $V \in \M_n$, then $V U V^{*}$ is a unitary extension of $B$ and
$\ker B = \vecspan\{V \vec{v}\}$.  Thus,
\begin{align*}
w_{B}(z) & =  \frac{z\langle (V U V^{*} - z I)^{-1} V \vec{v}, V \vec{v}\rangle}{\langle (V U V^{*} - z I)^{-1} V U V^{*} V \vec{v}, V \vec{v}\rangle}
 =  \frac{z\langle (U - z I)^{-1}  \vec{v},  \vec{v}\rangle}{\langle ( U - z I)^{-1} U  \vec{v}, \vec{v}\rangle}
 = w_{A}(z).
\end{align*}

For the other direction, we first give an alternate formula for $w_{A}$. 
Let $U \in \M_n$ be a unitary extension of $A$ and write
\begin{equation*}
U = Q \diag(\xi_1,\xi_2,\ldots,\xi_n) Q^*,
\end{equation*}
in which $Q \in \M_n$ is unitary and $|\xi_1| = |\xi_2| = \cdots = |\xi_n| = 1$.
Denote the $i$th entry of $\vec{q} = Q^{*} \vec{v}$ by $q_i$.  Then
\begin{equation*}
\langle (U - z I)^{-1} \vec{v}, \vec{v}\rangle = \sum_{j = 1}^{n} \frac{1}{\xi_j - z} |q_i|^2
\quad\text{and}\quad
\langle (U - z I)^{-1} U \vec{v}, \vec{v}\rangle = \sum_{j = 1}^{n} \frac{\xi_j}{\xi_j - z} |q_i|^2.
\end{equation*}
From here we see that 
\begin{align*}
w_{A}(z) & = \frac{z \langle (U - z I)^{-1} \vec{v}, \vec{v}\rangle}{\langle (U - z I)^{-1} U \vec{v}, \vec{v}\rangle}\\
& = \frac{\displaystyle \sum_{i = 1}^{n} \frac{z + \xi_i  + z - \xi_i}{z - \zeta_i} |q_i|^2}{\displaystyle\sum_{i = 1}^{n} \frac{z + \xi_i - z + \xi_i}{z - \xi_i} |q_i|^2}\\
& = \frac{\displaystyle\sum_{i = 1}^{n} \frac{z + \xi_i}{z - \xi_i} |q_i|^2 + \sum_{i = 1}^{n} |q_i|^2}{\displaystyle\sum_{i = 1}^{n} \frac{z + \xi_i}{z - \xi_i}  |q_i|^2 - \sum_{i = 1}^{n} |q_i|^2}\\
& = \frac{\displaystyle\sum_{i = 1}^{n} \frac{z + \xi_i}{z - \xi_i} |q_i|^2 + 1}{\displaystyle\sum_{i = 1}^{n} \frac{z + \xi_i}{z - \xi_i} |q_i|^2 - 1}
\end{align*}
since 
\begin{equation*}
\sum_{i = 1}^{n} |q_i|^2 = \|Q^{*} \vec{v}\|^2 = \|\vec{v}\|^2 = 1.
\end{equation*}
A similar formula holds for $w_B(z)$.  Now suppose that $w_{A} = \xi w_{B}$ for some $\xi \in \T$. By adjusting the unitary extension of either $A$ or $B$ we can assume that $w_A = w_B$ (an equality of rational functions). Thus,
$U_{A}$ and $U_{B}$ have the same eigenvalues and multiplicities, so
they are unitarily similar:  $X U_{A} X^{*} = U_{B}$ for some unitary matrix $X$. 
Since $X (\ker A)^{\perp} = (\ker B)^{\perp}$ and $U_{A}|_{(\ker A)^{\perp}} = U_{B}|_{(\ker B)^{\perp}}$, it follows that $X A X^{*} = B$ and hence $A$ and $B$ are unitarily similar. 
\end{proof}

\begin{Example}
The partial isometry
\begin{equation*}
A  = 
\begin{bmatrix}
 0 & -\frac{1}{\sqrt{2}} & \frac{1}{2} \\[3pt]
 0 & \frac{1}{\sqrt{2}} & \frac{1}{2} \\[3pt]
 0 & 0 & \frac{1}{\sqrt{2}} \\
\end{bmatrix}
\end{equation*}
is completely non-unitary since $\sigma(A) = \{0, \frac{1}{\sqrt{2}}\} \subset \D$.
A unitary extension for $A$ is 
\begin{equation*}
U = 
\begin{bmatrix}
 -\frac{1}{2} & -\frac{1}{\sqrt{2}} & \frac{1}{2} \\[3pt]
 -\frac{1}{2} & \frac{1}{\sqrt{2}} & \frac{1}{2} \\[3pt]
 \frac{1}{\sqrt{2}} & 0 & \frac{1}{\sqrt{2}} \\
\end{bmatrix}
\end{equation*}
and $\ker A = \vecspan\{ (1,0,0)\}$.
A computation using \eqref{Liv-Def} shows that 
\begin{equation*}
w_{A}(z) = - z \bigg(\frac{z - 1/\sqrt{2}}{1 - z/\sqrt{2}}\bigg)^2.
\end{equation*}
This is a finite Blaschke product whose zeros are the eigenvalues of $A$ with the corresponding multiplicities. 
\end{Example}

\begin{Example}
Consider the partially isometric matrix
\begin{equation*}
A =\begin{bmatrix}
 0 & 0 & 0 & 0 \\
 \frac{1}{\sqrt{2}} & -\frac{1}{\sqrt{2}} & 0 & 0 \\[3pt]
 \frac{1}{2} & \frac{1}{2} & 0 & 0 \\[3pt]
 \frac{1}{2} & \frac{1}{2} & 0 & 0 \\
\end{bmatrix}.
\end{equation*}
Since
$\sigma(A) = \{-\frac{1}{\sqrt{2}}, 0\} \subset \D$, we may apply Liv\v{s}ic's theorem.
Noting that $\ker A = \vecspan\{(0, 0, 0, 1), (0, 0, 1, 0)\}$, we see that
$A$ has unitary extension 
\begin{equation*}
U = \begin{bmatrix}
 0 & 0 & 1 & 0 \\
 \frac{1}{\sqrt{2}} & -\frac{1}{\sqrt{2}} & 0 & 0 \\[3pt]
 \frac{1}{2} & \frac{1}{2} & 0 & -\frac{1}{\sqrt{2}} \\[3pt]
 \frac{1}{2} & \frac{1}{2} & 0 & \frac{1}{\sqrt{2}} \\
\end{bmatrix}.
\end{equation*}
A computation with \eqref{Liv-Def} yields the $2 \times 2$ matrix-valued function
\begin{equation*}
w_{A}(z) = 
\begin{bmatrix}
 \frac{z}{\sqrt{2}} & \frac{z^2 \left(2 z+\sqrt{2}\right)}{2 \left(z+\sqrt{2}\right)} \\[4pt]
 -\frac{z}{\sqrt{2}} & \frac{z^2 \left(2 z+\sqrt{2}\right)}{2 \left(z+\sqrt{2}\right)} \\
\end{bmatrix}.
\end{equation*}
In particular, $w_{A}(\zeta)$ is unitary for every $\zeta \in \T$.
\end{Example}

\begin{Example}
Consider the partial isometry
$$A = 
\begin{bmatrix}
 0 & \frac{1}{2} & 0 & \frac{1}{2} & 0 \\[3pt]
 0 & 0 & 0 & 0 & 0 \\
 0 & \frac{1}{2} & 0 & \frac{1}{2} & 0 \\[3pt]
 0 & \frac{1}{\sqrt{2}} & 0 & -\frac{1}{\sqrt{2}} & 0 \\[3pt]
 0 & 0 & 0 & 0 & 0 \\
\end{bmatrix},$$
which has unitary extension 
$$U = 
\begin{bmatrix}
 0 & \frac{1}{2} & -\frac{1}{\sqrt{2}} & \frac{1}{2} & 0 \\[3pt]
 1 & 0 & 0 & 0 & 0 \\
 0 & \frac{1}{2} & \frac{1}{\sqrt{2}} & \frac{1}{2} & 0 \\[3pt]
 0 & \frac{1}{\sqrt{2}} & 0 & -\frac{1}{\sqrt{2}} & 0 \\[3pt]
 0 & 0 & 0 & 0 & 1 \\
\end{bmatrix}.$$
We have 
$\sigma(A) = \{-\frac{1}{\sqrt{2}}, 0\} \subset \D$ 
and 
$$\ker A = \vecspan\{(0, 0, 0, 0, 1), (0, 0, 1, 0, 0), (1, 0, 0, 0, 0)\}.$$
A computation with \eqref{Liv-Def} yields the $3 \times 3$ matrix-valued function 
\begin{equation*}
w_{A}(z) = 
\begin{bmatrix}
 z & 0 & 0 \\
 0 & \frac{z}{\sqrt{2}} & \frac{z^2 \left(2 z+\sqrt{2}\right)}{2 \left(z+\sqrt{2}\right)}
   \\[5pt]
 0 & -\frac{z}{\sqrt{2}} & \frac{z^2 \left(2 z+\sqrt{2}\right)}{2
   \left(z+\sqrt{2}\right)} \\
\end{bmatrix}.
\end{equation*}
As expected, $w_{A}(\zeta)$ is unitary for every $\zeta \in \T$.
\end{Example}

\begin{Example}
For the two matrices $A$ and $B$ from \eqref{ABHML},
\begin{equation*}
w_{A}(z) = 
\begin{bmatrix}
 z & 0 \\
 0 & z^3 \\
\end{bmatrix}
\qquad \text{and} \qquad
w_{B} = 
\begin{bmatrix}
 z^2 & 0 \\
 0 & z^2 \\
\end{bmatrix}.
\end{equation*}
There are no unitaries $Q_1, Q_2$ such that $w_{A}(z) = Q_1 w_{B}(z) Q_2$ for all $z \in \D$.  
If there were, then $ |z| = \| w_A (z) \| = \| w_B(z) \|  = |z|^2$ for all $z \in \D$, which is impossible.
\end{Example}

\section{The compressed shift}\label{Section:Compressed}

If a partial isometry $A \in \M_n$ satisfies 
\begin{equation}\label{eq:PIK1}
\sigma(A) = \{0, \lambda_1, \lambda_2, \ldots \lambda_{n - 1}\} \subset \D 
\qquad \text{and} \qquad \operatorname{\dim} \ker A = 1,
\end{equation} 
then there is a tangible representation of $A$ as a certain operator on a Hilbert space of rational functions.
What follows is a highly abbreviated treatment.  
See \cite{MR3526203, MR3793610} for the basics and \cite{N1,N2} for an encyclopedic treatment.

\subsection{A concrete model}
For a partial isometry $A \in \M_n$ that satisfies \eqref{eq:PIK1}, 
the \emph{model space} corresponding to $A$ is
\begin{equation*}
\K_A = \bigg\{\frac{a_{0} + a_{1} z + a_{2} z^2 + \cdots + a_{n - 1} z^{n - 1}}{(1  - \overline{\lambda_1} z) (1 - \overline{\lambda_2} z) \cdots (1 - \overline{\lambda_{n - 1}} z)}: a_j \in \C\bigg\}.
\end{equation*}
We endow $\K_A$ with a Hilbert-space structure by regarding it as a subspace of $L^2 = L^2(\T)$, with inner product 
\begin{equation*}
\langle f, g \rangle = \int_{\T} f(\zeta) \overline{g(\zeta)} dm(\zeta),
\end{equation*}
in which $dm(\zeta) = |d\zeta|/2 \pi$ is normalized Lebesgue measure on the unit circle $\T$. 
Let $H^2$ denote the Hardy space of analytic functions $f:\D\to\C$ with square-summable Taylor coefficients at the origin.
It can be viewed as a subspace of $L^2$ by considering boundary values on $\T$ (see \cite{MR3526203} and the references therein).
We associate to the partial isometry $A$, with data \eqref{eq:PIK1}, the $n$-fold Blaschke product
\begin{equation}\label{oijhboijhviuhvuhgv}
B_{A}(z) = z \prod_{j = 1}^{n - 1} \frac{z - \lambda_j}{1 - \overline{\lambda_j} z}.
\end{equation}
Then an exercise with the Cauchy integral formula confirms that $\K_{A} = H^2 \ominus B_{A} H^2$. 
Moreover, $\K_{A}$ is a reproducing kernel Hilbert space with kernel 
\begin{equation*}
k^{A}_{\lambda}(z) = \frac{1 - \overline{B_A(\lambda)} B_A(z)}{1 - \overline{\lambda} z}.
\end{equation*}
A convenient orthonormal basis for $\K_A$ is the \emph{Takenaka basis} \cite[Prop.~5.9.2]{MR3526203}.

\begin{Proposition}
Let $A \in \M_n$ be a partial isometry that satisfies \eqref{eq:PIK1}.  Then
\begin{align*}
v_1(z) &= 1,\\
v_2(z) &= z \frac{\sqrt{1 - |\lambda_1|^2}}{1 - \overline{\lambda_1} z},\\
v_3(z) &= z \frac{z - \lambda_1}{1 - \overline{\lambda_1} z} \frac{\sqrt{1 - |\lambda_2|^2}}{1 - \overline{\lambda_2} z},\\
v_4(z) &= z \frac{z - \lambda_1}{1 - \overline{\lambda_1} z} \frac{z - \lambda_2}{1 - \overline{\lambda_2} z} \frac{\sqrt{1 - |\lambda_3|^2}}{1 - \overline{\lambda_3} z},\\
&\,\,\,\vdots\\
v_{n}(z) &= z \bigg(\prod_{j = 1}^{n - 2} \frac{z - \lambda_j}{1 - \overline{\lambda_j} z}\bigg) \frac{\sqrt{1 - |\lambda_{n - 1}|^2}}{1 - \overline{\lambda_{n - 1}} z},
\end{align*}
is an orthonormal basis for $\K_{A}$. 
\end{Proposition}

The orthogonal projection $P_A:L^2\to L^2$ with range $\K_A$ is
\begin{equation*}
(P_{A} f)(z) = \sum_{i = 1}^{n} \langle f, v_i\rangle v_i(z).
\end{equation*}
This permits us to define the following operator.

\begin{Definition}
Let $A \in \M_n$ be a partial isometry that satisfies \eqref{eq:PIK1}.
The \emph{compressed shift} $S_A : \K_A\to\K_A$ is
\begin{equation*}
(S_{A} f)(z) = \sum_{i = 1}^{n} \langle z f, v_i\rangle v_{i}(z).
\end{equation*}
\end{Definition}

The operator $S_A$ enjoys the following properties; see \cite{MR3793610} for details.
It is a completely non-unitary partial isometry on $\K_A$ (that is, $S_A S_A^* S_A = S_A$).
Moreover, $\sigma(S_{A}) = \{0, \lambda_1, \lambda_2, \ldots, \lambda_{n - 1}\}$ and
$\ker S_A = \vecspan\{ B_A(z)/z\}$ is one dimensional.
The matrix representation of $S_A$ with respect to the Takenaka basis is
\begin{equation}\label{eq:modelHM}
\begin{bmatrix}
0 &  &  &  & \\
  & \lambda_1 &  &  & \\
    &  &  \ddots & \\
  &  q_{i, j} &  & \lambda_{n - 2} & \\
  &  &  &  & \lambda_{n - 1}\\
\end{bmatrix},
\end{equation}
in which
\begin{equation*}
q_{i,j} = \left(\prod_{k = i + 1}^{j - 1} (-\overline{\lambda_k})\right) \sqrt{1 - |\lambda_i|^2} \sqrt{1 - |\lambda_j|^2}.
\end{equation*}
In particular, $A$ is unitarily similar to \eqref{eq:modelHM} because they are partial isometries with one-dimensional kernels
and the same characteristic polynomials  (Theorem \ref{Theorem:HalmosUS}).

\begin{Proposition}\label{Proposition:ShiftPI}
If $A \in \M_n$ is a partial isometry that satisfies \eqref{eq:PIK1}, then $A$ is unitarily similar to \eqref{eq:modelHM}.
\end{Proposition}

Thus, the compressed shift is a model for certain types of partial isometries.

\subsection{Numerical range}\label{Section:Numerical}

The \emph{numerical range} of $A \in \M_n$ is 
\begin{equation*}
W(A) = \big\{\langle A \vec{x}, \vec{x}\rangle: \|\vec{x}\| = 1\big\}.
\end{equation*}

The continuity of $f(\vec{x})= \langle A \vec{x}, \vec{x}\rangle$, the compactness of the unit ball in $\C^n$, and the
Cauchy--Schwarz inequality ensure that $W(A)$ is a compact subset of $\{z \in \C: |z| \leq \|A\|\}$.
The Toeplitz--Hausdorff theorem says that $W(A)$ is convex \cite[Thm.~10.3.9]{MR3793610}.

The numerical range is unitarily invariant: if $A \cong B$, then $W(A) = W(B)$.  This permits
us to characterize the numerical range of a normal matrix using the following notions.
A \emph{convex combination} of
$\xi_1, \xi_2,\ldots, \xi_n \in \C$ is an expression
\begin{equation*}
c_1 \xi_1 + c_2 \xi_2 + \cdots + c_n \xi_n,
\end{equation*}
in which 
\begin{equation*}
c_1,c_2,\ldots, c_n \in [0,1] \qquad \text{and} \qquad 
c_1 + c_2 + \cdots + c_n = 1.
\end{equation*}
The \emph{convex hull} of $\{\xi_1,\xi_2,\ldots,\xi_n\}$
is the set of all convex combinations of $\xi_1,\xi_2,\ldots,\xi_n$.  It is the smallest filled polygon that contains
the points $\xi_1,\xi_2,\ldots,\xi_n$.

\begin{Proposition}\label{9UhYYt7}
The numerical range of a normal matrix is the convex hull of its eigenvalues. 
\end{Proposition}

\begin{proof}
Let $N \in \M_n$ be normal.  The spectral theorem ensures that $N$ is unitarily similar to 
a diagonal matrix $D = \diag(\xi_1, \xi_2, \ldots, \xi_n)$.  Thus,
\begin{align*}
W(N)
&= W(D) 
=\Big\{\langle D \vec{x}, \vec{x}\rangle: \|\vec{x}\|^2 = \sum_{i = 1}^{n} |x_i|^2 = 1\Big\}\\
& = \Big\{ \sum_{i = 1}^{n} \xi_i |x_i|^2: \sum_{i = 1}^{n} |x_i|^2 = 1\Big\} \\
& = \Big\{ \sum_{i = 1}^{n} c_i \xi_i : \sum_{i = 1}^{n} c_i = 1\Big\} 
\end{align*}
is the convex hull of $\{\xi_1,\xi_2,\ldots,\xi_n\}$.
\end{proof}

\begin{figure}
 \includegraphics[width=0.5\textwidth]{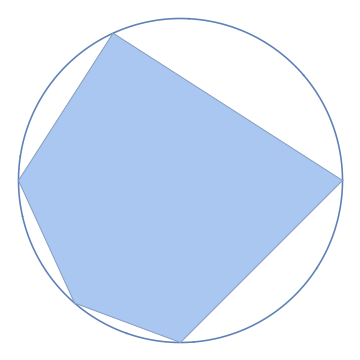}
 \caption{The numerical range of a unitary matrix with five eigenvalues on the unit circle is the convex hull of these eigenvalues.}
 \label{Figure:Five}
\end{figure}

Since the eigenvalues of a unitary matrix have unit modulus, the numerical range of a unitary matrix
is a polygon inscribed in the unit circle (Proposition \ref{9UhYYt7}); see Figure \ref{Figure:Five}.
For a partial isometry there is some beautiful geometry behind the scenes;
see the recent book of Daepp, Gorkin, Shaffer, and Voss \cite{Gorkin}.

A partial isometry $A \in \M_n$ with spectrum
$\{0,\lambda_1,\lambda_2,\ldots,\lambda_{n-1}\}$ and one-dimensional kernel, in which 
$\lambda_1,\lambda_2,\ldots,\lambda_{n-1} \in \D$, is unitarily similar to \eqref{eq:modelHM}.
The numerical range of $S_A$, and hence $W(A)$, can be computed as follows.
First consider the $(n+1)$-fold Blaschke product 
\begin{equation*}
b_A(z) = z^2 \prod_{j = 1}^{n - 1} \frac{z - \lambda_j}{1 - \overline{\lambda_j} z}.
\end{equation*}
For each $\xi \in \T$, there are $n + 1$ distinct points 
$\zeta_1, \zeta_2, \ldots, \zeta_{n + 1} \in \T$
such that $b_{A}(\zeta_i) = \xi$ for $1 \leq i \leq n + 1$ \cite[p.~48]{MR3793610}. Let
$Q_{\xi}$ denote the convex hull of $\co(\{\zeta_1, \zeta_2, \ldots, \zeta_{n + 1}\})$,
which is a $(n + 1)$-gon whose vertices are on $\T$.  Then,
\begin{equation}\label{HyGGYFRQOwq}
W(A) = \bigcap_{\xi \in \T} Q_{\xi}.
\end{equation}

\begin{Example}\label{Example:Shift2}
The partial isometry
\begin{equation*}
A = \begin{bmatrix}
0&0\\
\frac{\sqrt{3}}{2}&\frac{1}{2}\\
\end{bmatrix}
\end{equation*}
has $\sigma(A) = \{0, \frac{1}{2}\}$.  Then 
\begin{equation*}
b_{A}(z) = z^2 \bigg(\frac{z - 1/2}{1 - z/2}\bigg).
\end{equation*}
The sets $Q_{\xi}$ are filled triangles; see Figure \ref{Figure:21}.
The numerical range $W(A)$ of $A$ is the intersection of the $Q_{\xi}$,
which is an ellipse; see Figure \ref{Figure:22}.  See \cite{Gorkin} for more on the specifics about the ellipse.
\end{Example}

            \begin{figure}
		\centering
		\begin{subfigure}[b]{0.475\textwidth}
	                \centering
	                \includegraphics[width=\textwidth]{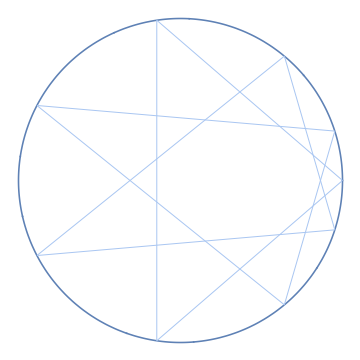}
                         \caption{The triangles $Q_{\xi}$ for $\xi = e^{2 \pi i k/3}$ with $k = 0, 1, 2$.}
	                \label{Figure:21}
	        \end{subfigure}
	        \quad
		\begin{subfigure}[b]{0.475\textwidth}
	                \centering
	                \includegraphics[width=\textwidth]{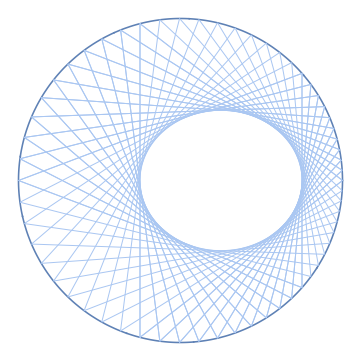}
	                \caption{$W(A)$ is the intersection of the triangles $Q_{\xi}$ for $\xi \in \T$.}
	                \label{Figure:22}
	        \end{subfigure}
	        \caption{Illustration for Example \ref{Example:Shift2}}
	      \end{figure}

\begin{Example}\label{Example:Shift3}
Consider the compressed shift $A$ with 
\begin{equation*}
\sigma(A) = \{0, \tfrac{1}{\sqrt{2}}, \tfrac{1}{\sqrt{3}}, \tfrac{1}{\sqrt{5}}\}.
\end{equation*}
The associated five-fold Blaschke product is 
\begin{equation*}
b_{A}(z) = z^2 \bigg(\frac{z - 1/\sqrt{2}}{1 - z/\sqrt{2}}\bigg)\bigg( \frac{z - 1/\sqrt{3}}{1 - z/\sqrt{3}} \bigg)\bigg(\frac{z - 1/\sqrt{5}}{1 - z/\sqrt{5}} \bigg).
\end{equation*}
The sets $Q_{\xi}$ are (irregular) pentagons; see Figure \ref{Figure:31}.
The numerical range $W(A)$ of $A$ is the intersection of the $Q_{\xi}$, which is an ellipse; see Figure \ref{Figure:32}. 
\end{Example}

A result that relates Corollary \ref{Proposition:ShiftPI} and 
the Halmos--McLaughlin theorem on defect-one partial isometries
(Theorem \ref{Theorem:HalmosUS}) is the following \cite{MR3099197, MR3134275}. 

\begin{Theorem}
If $A, B \in \M_n$ are partial isometries with spectra contained in $\D$ and one-dimensional kernels, then 
$A \cong B$ if and only if $W(A) = W(B)$. 
\end{Theorem}

\begin{figure}
		\centering
		\begin{subfigure}[b]{0.475\textwidth}
	                \centering
	                \includegraphics[width=\textwidth]{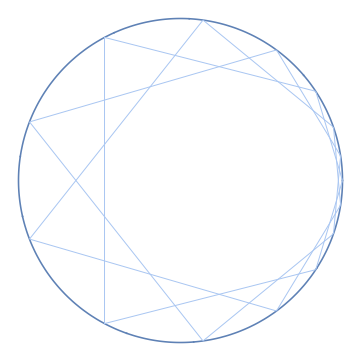}
	                \caption{The pentagons $Q_{\xi}$ for $\xi = e^{2 \pi i k/3}$ with $k = 0, 1, 2$.}
	                \label{Figure:31}
	        \end{subfigure}
	        \quad
		\begin{subfigure}[b]{0.475\textwidth}
	                \centering
	                \includegraphics[width=\textwidth]{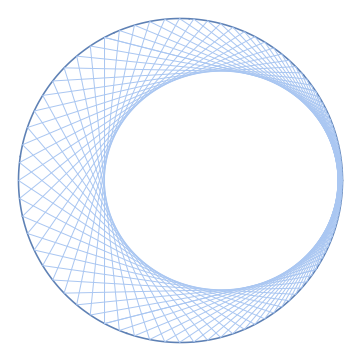}
	                \caption{$W(A)$ is the intersection of the pentagons $Q_{\xi}$ for $\xi \in \T$.}
	                \label{Figure:32}
	        \end{subfigure}
	        \caption{Illustration for Example \ref{Example:Shift3}}
	      \end{figure}

\bibliographystyle{plain}

\bibliography{PIM_references}

\def\Dbar{\leavevmode\lower.6ex\hbox to 0pt{\hskip-.23ex \accent"16\hss}D}
  \def\Dbar{\leavevmode\lower.6ex\hbox to 0pt{\hskip-.23ex \accent"16\hss}D}
  \def\Dbar{\leavevmode\lower.6ex\hbox to 0pt{\hskip-.23ex \accent"16\hss}D}
\begin{thebibliography}{10}

\bibitem{Gorkin}
U.~Daepp, P.~Gorkin, A.~Shaffer, and K.~Voss.
\newblock {\em Finding Ellipses}, volume~34 of {\em Carus Mathematical
  Monographs}.
\newblock American Mathematical Society, 2018.

\bibitem{Djokovic}
Dragomir~{\v{Z}}. Djokovi{\'c}.
\newblock Poincar\'e series of some pure and mixed trace algebras of two
  generic matrices.
\newblock {\em J. Algebra}, 309(2):654--671, 2007.

\bibitem{Erd}
Ivan Erd{\'e}lyi.
\newblock On partial isometries in finite-dimensional {E}uclidean spaces.
\newblock {\em SIAM J. Appl. Math.}, 14:453--467, 1966.

\bibitem{MR0255557}
Ivan Erd\'{e}lyi.
\newblock Partial isometries and generalized inverses.
\newblock In {\em Proc. {S}ympos. {T}heory and {A}pplication of {G}eneralized
  {I}nverses of {M}atrices ({L}ubbock, {T}exas, 1968)}, pages 203--217. Texas
  Tech. Press, Lubbock, Tex., 1968.

\bibitem{ErdelyiProduct}
Ivan Erd\'{e}lyi.
\newblock Partial isometries closed under multiplication on {H}ilbert spaces.
\newblock {\em J. Math. Anal. Appl.}, 22:546--551, 1968.

\bibitem{GarciaHorn}
Stephan~Ramon Garcia and Roger~A. Horn.
\newblock {\em A second course in linear algebra}.
\newblock Cambridge University Press, Cambridge, 2017.

\bibitem{MR3526203}
Stephan~Ramon Garcia, Javad Mashreghi, and William~T. Ross.
\newblock {\em Introduction to model spaces and their operators}, volume 148 of
  {\em Cambridge Studies in Advanced Mathematics}.
\newblock Cambridge University Press, Cambridge, 2016.

\bibitem{MR3793610}
Stephan~Ramon Garcia, Javad Mashreghi, and William~T. Ross.
\newblock {\em Finite {B}laschke products and their connections}.
\newblock Springer, Cham, 2018.

\bibitem{CSO}
Stephan~Ramon Garcia and Mihai Putinar.
\newblock Complex symmetric operators and applications.
\newblock {\em Trans. Amer. Math. Soc.}, 358(3):1285--1315 (electronic), 2006.

\bibitem{CSO2}
Stephan~Ramon Garcia and Mihai Putinar.
\newblock Complex symmetric operators and applications. {II}.
\newblock {\em Trans. Amer. Math. Soc.}, 359(8):3913--3931 (electronic), 2007.

\bibitem{MSPI}
Stephan~Ramon Garcia and David Sherman.
\newblock Matrices similar to partial isometries.
\newblock {\em Linear Algebra Appl.}, 526:35--41, 2017.

\bibitem{UET}
Stephan~Ramon Garcia and James~E. Tener.
\newblock Unitary equivalence of a matrix to its transpose.
\newblock {\em J. Operator Theory}, 68(1):179--203, 2012.

\bibitem{CSPI}
Stephan~Ramon Garcia and Warren~R. Wogen.
\newblock Complex symmetric partial isometries.
\newblock {\em J. Funct. Anal.}, 257(4):1251--1260, 2009.

\bibitem{SNCSO}
Stephan~Ramon Garcia and Warren~R. Wogen.
\newblock Some new classes of complex symmetric operators.
\newblock {\em Trans. Amer. Math. Soc.}, 362(11):6065--6077, 2010.

\bibitem{MR3099197}
Hwa-Long Gau and Pei~Yuan Wu.
\newblock Numerical ranges and compressions of {$S_n$}-matrices.
\newblock {\em Oper. Matrices}, 7(2):465--476, 2013.

\bibitem{MR3134275}
Hwa-Long Gau and Pei~Yuan Wu.
\newblock Structures and numerical ranges of power partial isometries.
\newblock {\em Linear Algebra Appl.}, 440:325--341, 2014.

\bibitem{Halmos}
P.~R. Halmos and J.~E. McLaughlin.
\newblock Partial isometries.
\newblock {\em Pacific J. Math.}, 13:585--596, 1963.

\bibitem{HalmosLA}
Paul~R. Halmos.
\newblock {\em Linear algebra problem book}, volume~16 of {\em The Dolciani
  Mathematical Expositions}.
\newblock Mathematical Association of America, Washington, DC, 1995.

\bibitem{MR0227194}
John~Z. Hearon.
\newblock Partially isometric matrices.
\newblock {\em J. Res. Nat. Bur. Standards Sect. B}, 71B:225--228, 1967.

\bibitem{MR0225790}
John~Z. Hearon.
\newblock Polar factorization of a matrix.
\newblock {\em J. Res. Nat. Bur. Standards Sect. B}, 71B:65--67, 1967.

\bibitem{HornAlfred}
A.~Horn.
\newblock On the eigenvalues of a matrix with prescribed singular values.
\newblock {\em Proc. Amer. Math. Soc.}, 5:4--7, 1954.

\bibitem{HornJohnson}
Roger~A. Horn and Charles~R. Johnson.
\newblock {\em Matrix analysis}.
\newblock Cambridge University Press, Cambridge, second edition, 2013.

\bibitem{MR977922}
Kung~Hwang Kuo and Pei~Yuan Wu.
\newblock Factorization of matrices into partial isometries.
\newblock {\em Proc. Amer. Math. Soc.}, 105(2):263--272, 1989.

\bibitem{Livsic}
M.~S. Liv{\v{s}}ic.
\newblock Isometric operators with equal deficiency indices, quasi-unitary
  operators.
\newblock {\em Amer. Math. Soc. Transl. (2)}, 13:85--103, 1960.

\bibitem{Murnaghan}
Francis~D. Murnaghan.
\newblock On the unitary invariants of a square matrix.
\newblock {\em Anais Acad. Brasil. Ci.}, 26:1--7, 1954.

\bibitem{N2}
N.~Nikolski.
\newblock {\em Operators, functions, and systems: an easy reading. {V}ol. 2},
  volume~93 of {\em Mathematical Surveys and Monographs}.
\newblock American Mathematical Society, Providence, RI, 2002.
\newblock Model operators and systems, Translated from the French by Andreas
  Hartmann and revised by the author.

\bibitem{N1}
Nikolai~K. Nikolski.
\newblock {\em Operators, functions, and systems: an easy reading. {V}ol. 1},
  volume~92 of {\em Mathematical Surveys and Monographs}.
\newblock American Mathematical Society, Providence, RI, 2002.
\newblock Hardy, Hankel, and Toeplitz, Translated from the French by Andreas
  Hartmann.

\bibitem{Pearcy}
Carl Pearcy.
\newblock A complete set of unitary invariants for {$3\times 3$} complex
  matrices.
\newblock {\em Trans. Amer. Math. Soc.}, 104:425--429, 1962.

\bibitem{Pearcy2}
Carl Pearcy.
\newblock A complete set of unitary invariants for operators generating finite
  {$W\sp*$}-algebras of type {I}.
\newblock {\em Pacific J. Math.}, 12:1405--1416, 1962.

\bibitem{Sibirskii}
K.~S. Sibirski{\u\i}.
\newblock A minimal polynomial basis of unitary invariants of a square matrix
  of order three.
\newblock {\em Mat. Zametki}, 3:291--295, 1968.

\bibitem{Specht}
Wilhelm Specht.
\newblock Zur {T}heorie der {M}atrizen. {II}.
\newblock {\em Jber. Deutsch. Math. Verein.}, 50:19--23, 1940.

\bibitem{Weyl}
H.~Weyl.
\newblock Inequalities between the two kinds of eigenvalues of a linear
  transformation.
\newblock {\em Proc. Nat. Acad. Sci. U. S. A.}, 35:408--411, 1949.

\end{thebibliography}

\end{document}